\documentclass[a4paper,11pt]{article}
\usepackage{amssymb,amsmath,amsthm}
\usepackage{setspace}
\usepackage{graphicx} 
\usepackage{enumerate}
\usepackage{xcolor}

\usepackage{pgf,tikz}
\usepackage{mathrsfs}
\usetikzlibrary{arrows}
\usetikzlibrary[patterns]
\tikzstyle{dashed}=[dash pattern=on 3pt off 3pt]

\usepackage{tikz-cd}
\usetikzlibrary{%
  matrix,%
  calc,%
  arrows%
}

\usepackage[backend=bibtex, style=alphabetic,maxbibnames=99]{biblatex}
\usepackage{filecontents}
\usepackage{sectsty}

\theoremstyle{plain}
\newtheorem{thm}{Theorem}[section]
\newtheorem{prop}[thm]{Proposition}
\newtheorem{coro}[thm]{Corollary}
\newtheorem{lem}[thm]{Lemma}

\theoremstyle{definition}
\newtheorem{dfn}[thm]{Definition}

\newtheorem{rmk}[thm]{Remark}

\allsectionsfont{\centering}

\providecommand{\lan}{\mathcal{L}}

\providecommand{\M}{\mathcal{M}}
\providecommand{\Mlin}{\mathcal{M}^{<}}
\providecommand{\cl}{\mathcal{C}}
\providecommand{\rat}{\mathbb{Q}}

\def\Ind#1#2{#1\setbox0=\hbox{$#1x$}\kern\wd0\hbox to 0pt{\hss$#1\mid$\hss}
\lower.9\ht0\hbox to 0pt{\hss$#1\smile$\hss}\kern\wd0}
\def\Notind#1#2{#1\setbox0=\hbox{$#1x$}\kern\wd0\hbox to 0pt{\mathchardef
\nn="3236\hss$#1\nn$\kern1.4\wd0\hss}\hbox to 0pt{\hss$#1\mid$\hss}\lower.9\ht0
\hbox to 0pt{\hss$#1\smile$\hss}\kern\wd0}
\def\ind{\mathop{\mathpalette\Ind{}}}

\newcommand{\indst}{\ind^{\star}}

\begin{document}

\title{Automorphism groups of linearly ordered homogeneous structures}

\author{Yibei Li}

\newcommand{\Addresses}{{
  \bigskip
  \footnotesize

  Yibei Li, \textsc{Department of Mathematics, Imperial College London,
    Huxley Building, 180 Queens Gate, London SW7 2AZ}\par\nopagebreak
  \textit{E-mail address}: \texttt{yibei.li16@imperial.ac.uk}

}}

\date{}

\maketitle

\section*{Abstract}
We apply results proved in \cite{li2019automorphism}, which is generalised from arguments in \cite{tentziegler2013isometry}, to the linear order expansions of non-trivial free homogeneous structures and the universal $n$-linear order for $n \geq 2$ and prove the simplicity of their automorphism groups.

\section{Introduction}

Given a relational language $\lan$, a countable $\lan$-structure $\M$ is \emph{homogeneous} if every partial isomorphism between finite substructures of $\M$ extends to an automorphism of $\M$. Fra{\"\i}ss{\'e}'s Theorem \cite{fraisse1953theorem} provides one way of constructing homogeneous structures by establishing a one-to-one correspondence between such structures and amalgamation classes. We call the homogeneous structure the \emph{Fra{\"\i}ss{\'e} limit} of the corresponding amalgamation class.


A special type of amalgamation classes is the free amalgamation class and we say a homogeneous structure is \emph{free} if it is the Fra{\"\i}ss{\'e} limit of a free amalgamation class. Examples include the random graph, the universal $K_n$-free graphs, etc. In \cite{macphersontent2011simplicity}, Macpherson and Tent proved the following theorem about free homogeneous structures using ideas and results from model theory and topological groups:

\begin{thm}\label{tztheorem}(\cite{macphersontent2011simplicity})
Let $\M$ be a countable free homogeneous relational structure. Suppose $Aut(\M) \neq Sym(\M)$ and $Aut(\M)$ is transitive on $\M$. Then $Aut(\M)$ is simple.
\end{thm}

This is then generalised by Tent and Ziegler \cite{tentziegler2013isometry} to a homogeneous structure with a \emph{stationary independence relation} (see Definition \ref{swir}), which is weaker than a free homogeneous structure. They applied their result to the Urysohn space, which has a local stationary independence relation, but is not free. The author \cite{li2018simplicity} applied their result to some undirected graphs constructed by Cherlin \cite{cherlin1998classification}. However, we cannot apply the result to directed graphs as they do not satisfy the symmetry axiom of the stationary independence relation. Hence, the author \cite{li2019automorphism} then generalised the notion of a stationary indpendence relation to one without the symmetry axiom as the following.

\begin{dfn}\label{swir}
Let $\M$ be a homogeneous structure and suppose $A\ind _B C$ is a ternary relation between finite substructures $A,B,C$ of $\M$. We say that $\ind$ is a \emph{stationary weak independence relation} ($\mathbf{SWIR}$) if the following axioms are satisfied:
\begin{enumerate}[(i)]
\item Invariance: for any $g \in Aut(\M)$, if $A \ind _B C$, then $ gA \ind _{gB} gC$

\item Monotonicity: $A \ind _B CD \Rightarrow A \ind _B C$, $A \ind _{BC} D$

\hspace{2.4cm} $AD \ind _B C \Rightarrow A \ind _B C$, $D \ind _{AB} D$

\item Transitivity: $A \ind _B C$, $A \ind _{BC} D \Rightarrow A \ind _B CD $

\hspace{2.1cm} $A \ind _B C$, $D \ind _{AB} C \Rightarrow AD \ind _B C $

\item Existence: If $p$ is an $n$-type over $B$ and $C$ is a finite set, then $p$ has realisation $\bar{a}$, $\bar{a}'$ such that $\bar{a} \ind _B C$ and $C \ind _B \bar{a}'$.

\item Stationarity: If $\bar{a}$ and $\bar{a'}$ are $n$-tuples that have the same type over $B$ and $\bar{a} \ind _B C$, $\bar{a'} \ind _B C$, then $\bar{a}$ and $\bar{a'}$ have the same type over $BC$.

\hspace{2.3cm} If $\bar{a}$ and $\bar{a'}$ are $n$-tuples that have the same type over $B$ and $C \ind _B \bar{a}$, $C \ind _B \bar{a'}$, then $\bar{a}$ and $\bar{a'}$ have the same type over $BC$.

\end{enumerate}

If in addition, $\M$ satisfies symmetry, i.e. $A\ind _B C  \Rightarrow C \ind _B A$, then we say $\ind$ is a stationary independence relation.

\end{dfn}

In \cite{li2019automorphism}, the author generalised Tent and Ziegler's result to the following theorem using the following notion and applied it to some directed graphs constructed by Cherlin \cite{cherlin1998classification}. 

\begin{dfn}
We say that $g \in Aut(\M)$ \emph{moves almost $R$-maximally} if for any finite set $X$ and $n$-type $p$ over $X$, there is a realisation $\bar{a}$ of $p$ such that $\bar{a} \ind _{X} g\bar{a}$.

We say $g \in Aut(\M)$ \emph{moves almost $L$-maximally} if for any finite set $X$ and $n$-type $p$ over $X$, there is a realisation $\bar{a}$ of $p$ such that $g\bar{a}  \ind _{X} \bar{a}$.
\end{dfn}

\begin{thm}\label{swirthm}
Suppose $\M$ is a countable structure with a $\mathbf{SWIR}$ and $g \in Aut (\M)$ is such that $g$ moves almost $R$-maximally and $g^{-1}$ moves almost $L$-maximally. Then any element of $G$ is a product of conjugates of $g$.
\end{thm}

In this paper, we will show that this theorem can be applied to the linear order expansions (see Definition \ref{linearorderexp}) of some structures and prove the following theorem:

\begin{thm}\label{main}
Let $\M$ be one of the following structures:
\begin{enumerate}[(i)]
\item the linear order expansion of a non-trivial free homogeneous structure
\item the universal $n$-linear order for $n\geq 2$.
\end{enumerate}
Then $Aut(\M)$ is simple.
\end{thm}

In Section 2, we prove some results about general linearly ordered structures, which will be used in the proofs later. In Section 3, we apply Theorem \ref{main} to the linear order expansions of free homogeneous structures. The same result was shown in \cite{ckt2019simplicity} recently using a somewhat similar approach. We then apply the theorem to the universal $n$-linear orders in Section 4. Note that when $n$=1, the structure is the dense linear order $(\rat, \leq)$ and its automorphism group was studied in \cite{holland1963lattice} and \cite{lloyd1964lattice}. We only provide an alternative approach in this case.

Throughout this paper, we let $\lan$ be a relational language and we will only consider countable homogeneous structure $\M$ whose age $\cl$ satisfies the strong amalgamation property. Equivalently, the algebraic closure in $\M$ is trivial, i.e. $acl(A)=A$ for all finite substructure $A$ of $\M$. A countable homogeneous structure $\M$ with age $\cl$ satisfies the \emph{Extension Property}, i.e. if $A \subseteq \M$ is finite and $f : A \rightarrow B$ is an embedding and $B\in \cl$, then there is an embedding $g : B \rightarrow \M$ such that $g(f(a)) = a$ for all $a \in A$.

\section{General linearly ordered structures}
We first define the \emph{linear order expansion} of a structure.

\begin{dfn}\label{linearorderexp}
Let $\cl$ be a strong amalgamation class and $\M$ be its Fra{\"\i}ss{\'e} limit. Let $\lan^{<}=\lan \cup \{<\}$, where the new relational symbol $<$ is interpreted as a (strict) linear order. For any $A \in \cl$, we can put a linear order on $A$. Then $A$ is an $\lan^{<}$-structures. We call it the \emph{linear order expansion} of $A$. Let $\cl^{<}$ be the class of all isomorphism types of such expansions. Then $\cl^{<}$ is an amalgamation class. We call its Fra{\"\i}ss{\'e} limit, denoted by $\M ^{<}$, the \emph{linear order expansion} of $\M$.
\end{dfn}

\begin{rmk}\label{expansionrmk}
\begin{enumerate}[(i)]
\item We can check that $\cl^{<}$ is an amalgamation class. It satisfies the amalgamation property because for any finite $B \subseteq A,C \subseteq \Mlin$, let $A^{\lan},B^{\lan},C^{\lan}$ be the $\lan$-reducts of $A,B,C$. Since $\cl$ is a strong amalgamation class, we can find $D^{\lan} \in \cl$ such that $A^{\lan},C^{\lan}$ can be embedded into $D^{\lan}$ such that $A^{\lan}\cap C^{\lan}=B^{\lan}$. Since $\cl^{\lan}$ contains all isomorphism types of the linear order expansion of $D^{\lan}$, we can find $D \in \cl^{\lan}$ satisfying that for $a \in A \setminus B, c \in C \setminus B$, $a<c$ if and only if there exists $b\in B$ such that $a<b<c$. Then $D$ is the amalgamation of $A,C$ over $B$.
\item Also note that $\Mlin$ satisfies the theory of the dense linear order. Hence, we can identify elements of $\Mlin$ as elements of $\rat$ and it makes sense to have intervals of $\Mlin$.
\item \cite{ckt2019simplicity} defined a more general notion. Given disjoint relational languages $\lan_1, \lan_2$, suppose $\M_i$ is a homogeneous $\lan_i$-structure for $i=1,2$ and $\M_1,\M_2$ have the same underlying set $\M$. Let $\lan=\lan_1\cup \lan_2$. Then an $\lan$-structure $\M^{\star}$ on $\M$ is called a \emph{free fusion} of $\M_1$ and $\M_2$ if
\begin{enumerate}[(a)]
\item the $\lan_i$-reduct of $\M^{\star}$ is $\M_i$ for $i=1,2$ and 
\item for every non-algebraic $\lan_i$-type $p_i$ for $i = 1, 2$, their union $p_1 \cup p_2$ is realised in $\M^{\star}$.
\end{enumerate}
Then, by definition and part (iv) of this remark, the linear order expansion of $\M$ is the free fusion of $\M$ and the dense linear order. In \cite{lp2014permutations}, $\Mlin$ is called a \emph{superposition} of $\M$ and the dense linear order.
\end{enumerate}
\end{rmk}

In order to prove the next theorem, we define the following notations.

\begin{dfn}
Let $\lan$ be a relational language and $\lan^{<}=\lan \cup \{<\}$. Let $\M$ be an $\lan$-structure and $\Mlin$ its linear order expansion. For $X \subseteq \Mlin$, let $tp(\bar{x}/X)$ be a $n$-type over $X$. We define $tp^{\lan}(\bar{x}/X)$ to be the set of all $\lan$-formulas satisfied by $\bar{x}$ with parameters in $X$ and $tp^{<}(\bar{x}/X)$ to be the set of all $\{<\}$-formulas satisfied by $\bar{x}$ with parameters in $X$. We call them \emph{$\lan$-type} and \emph{$\{<\}$-type} respectively.
\end{dfn}

\begin{rmk}\label{septypermk}
It is straightforward to see $tp(\bar{a}/X)=tp(\bar{a}'/X)$ if and only if $tp^{<}(\bar{a}/X)=tp^{<}(\bar{a}'/X)$ and $tp^{\lan}(\bar{a}/X)=tp^{\lan}(\bar{a}'/X)$. Also note that for any $n$-$\lan$-type $p(\bar{x})$ and any $n$-$\{<\}$-type $q(\bar{x})$, $p(x) \cup q(x)$ is consistent.
\end{rmk}

\begin{thm}\label{swirfromswir}
Let $\lan$ be a relational language and $\lan^{<}=\lan \cup \{<\}$. Let $\M$ be a countable homogeneous structure with a $\mathbf{SWIR}$ $\ind^{\star}$. For any finite subsets $A,B,C$ of its linear order expansion $\M ^{<}$, we define $A \ind _B C$ if $A \ind^{\star}_B C$ and for any $a \in A \setminus B, c\in C \setminus B$, there exists $b\in B$ such that $a<b<c$. Then $\ind$ is a $\mathbf{SWIR}$ on $\M ^{<}$.
\end{thm}
\begin{proof}

\underline{Monotonicity:} Suppose $A\ind _B CD$. Then $A \ind^{\star}_B CD $. By Monotonicy of $\ind^{\star}$, we have $A \indst _B C$ and $A \indst_{BC} D$. We also have that for $a\in A\setminus B, c \in C \setminus B$, there exists $b\in B$ such that $a<b<c$. For any $a \in A \setminus BC, d \in D\setminus BC$, there exists $b\in B$ such that $a <b<d$. Therefore, we have $A \ind _B C$ and $A \ind_{BC} D$.

\underline{Transitivity:} Suppose $A \ind _B C$ and $A \ind_{BC} D$. By Transitivity of $\indst$, we have $A\indst _B CD$. For any $a\in A \setminus B$, $d \in D\setminus B$ such that $a<d$, there exists $b \in B$ such that $a<b<d$ or there exists $c \in C\setminus B$ such that $a<c<d$. In the latter case, there exists $b\in B$ such that $a<b<c<d$. Hence, in both cases, there exists $b\in B$ such that $a <b<d$. Therefore, we have $A\ind _B CD$.

\underline{Stationarity:} Suppose $tp(\bar{a}/B)=tp(\bar{a}'/B)$ and $\bar{a} \ind _B C$, $\bar{a}' \ind_B C$. Then $tp^{\lan}(\bar{a}/B)=tp^{\lan}(\bar{a}'/B) $ and $\bar{a} \ind^{\star}_B C$, $\bar{a}' \ind^{\star}_B C$. By Stationarity of $\ind^{\star}$, we have $tp^{\lan}(\bar{a}/BC)=tp^{\lan}(\bar{a}'/BC)$. To see that $tp^{<}(\bar{a}/BC)=tp^{<}(\bar{a}'/BC)$, suppose otherwise. Then, there exists $a_i \in \bar{a}, a'_i \in \bar{a}'$  and $c \in C \setminus B$ such that $a_i<c<a'_i$. Since $\bar{a} \ind_B C$, there exists $b\in B$ such that $a_i<b<c<a'_i$, which contradicts $tp(\bar{a}/B)=tp(\bar{a}'/B)$. Therefore, we have $tp^{<}(\bar{a}/BC)=tp^{<}(\bar{a}'/BC)$ and thus, $tp(\bar{a}/BC)=tp(\bar{a}'/BC)$.

\underline{Existence:} Let $p=tp(x/B)$ be an $n$-type. Since by part (i) of Remark \ref{expansionrmk}, we can find an amalgamation of $\bar{a}'B$ and $BC$ over $B$ where $\bar{a}'$ realises $p$ and for $a_i \in \bar{a}' \setminus B, c \in C \setminus B$, we have $a<c$ if and only if there exists $b\in B$ such that $a<b<c$. We can embed $BC$ into this amalgamation. Then by the Extension Property, we can embed this amalgamation back to $\Mlin$. Hence, we can find $\bar{a} \in \Mlin$ realising $p$ such that $\bar{a} \ind _B C$.

Invariance is straightforward to see. We can prove the remaining by swapping the sides of $\ind$.
\end{proof}

In order to prove our main theorem, we define the following notions. 

\begin{dfn}
We say an automorphism $g$ of $\Mlin$ is \emph{right-bounded} if there exist $a \in \Mlin$ such that $gb =b$ for any $b > a$. It is \emph{left-bounded} if there exist $a \in \Mlin$ such that $gb =b$ for any $b < a$. It is \emph{unbounded} if it is neither left-bounded nor right-bounded. 
\end{dfn}


\begin{dfn}
Let $S$ be a subset of $\Mlin$. We define the \emph{convexification} of $S$ to be the set $\{ a \in \Mlin| \exists s_1,s_2 \in S: s_1 \leq a \leq s_2 \}$. Let $g$ be an automorphism of $\Mlin$. An \emph{orbital} of $g$ containing some $a \in \Mlin$ is the convexification of $\{g^n a| n \in \mathbb{Z} \}$. We say an orbital $I$ is \emph{unbounded above} if for any $a \in \Mlin$, there exists $b >a$ such that $b \in I$ and we say it is \emph{unbounded below} if for any $a \in \Mlin$, there exists $b <a$ such that $b \in I$.
\end{dfn}
Note that if $ga >a$, then $gb>b$ for all $b$ in the orbital of $g$ containing $a$. Similarly, if $ga <a$, then $gb<b$ for all $b$ in the orbital of $g$ containing $a$. So we can define the following:
\begin{dfn}
\begin{enumerate}[(i)]
\item We say an orbital $I$ is a \emph{$+$-orbital} of $g$ if $ga>a$ for all $a\in I$ and is a \emph{$-$-orbital} of $g$ if $ga<a$ for all $a\in I$.
\item For $g \in Aut(\Mlin)$, we say $g$ \emph{has a single orbital} if there exists $a \in \Mlin$ such that the sequence $(g^ia)_{i\in \mathbb{Z}}$ is unbounded above and below.  
\end{enumerate}
\end{dfn}

\begin{rmk}\label{orbitalrmk}
\begin{enumerate}[(i)]
\item If $g \in Aut(\Mlin)$ has $+$-orbital unbounded above (or below), then for any $b$ in the orbital, $(g^ib)_{i\in \mathbb{Z}}$ is a sequence unbounded above (or below). 
\item Similarly if $g \in Aut(\Mlin)$ has $-$-orbital unbounded above (or below), then for any $b$ in the orbital, $(g^ib)_{i\in \mathbb{Z}}$ is a sequence unbounded above (or below). 
\item Thus, if $g \in Aut(\Mlin)$ has a single $+$-orbital or a single $-$-orbital, then for any $b \in Aut(\Mlin)$, $(g^ib)_{i\in \mathbb{Z}}$ is a sequence unbounded above and below.
\end{enumerate}
\end{rmk}

\begin{lem}\label{epininterval}
Let $\M$ be a countable homogeneous $\lan$-structure and $\Mlin$ be its linear order expansion. Let $p(x)$ be a 1-tpye over some finite set $X\subseteq \Mlin$ and $b,c \in \Mlin$. Suppose $p^{<} (x)  \cup \{ b<x<c \}$ is consistent. Then there exists $a\in (b,c)$ realising $p(x)$.
\end{lem}
\begin{proof}
The lemma follows from the Extension Property. Since $p^{<}(x) \cup \{ b<x<c \}$ is consistent, we can embed $bcX$ into $abcX$ such that $a$ satisfies $p(x)$ and $b<a<c$. Then by the Extension Property, we can embed $abcX$ into $\Mlin$. Hence, there exists $a\in (b,c)$ realising $p(x)$.
\end{proof}

\begin{coro}\label{exofreal}
Let $\M$ be a countable homogeneous $\lan$-structure and $\Mlin$ be its linear order expansion. Let $p(x)$ be a 1-tpye over some finite set $X\subseteq \Mlin$ and $b \in \Mlin$. Let $g\in Aut(\Mlin)$.
\begin{enumerate}[(i)]
\item If for all $y \in \Mlin$, there exists $z>y$ such that $gz>z$. Suppose $p^{<} (x) \vdash \{ x>x' :x' \in X\}$. Then for any $y \in \Mlin$, there exists $a>y$ such that $a$ realises $p(x)$ and $ga>a$. 
\item If for all $y \in \Mlin$, there exists $z>y$ such that $gz<z$. Suppose $p^{<} (x) \vdash \{ x>x' :x' \in X\}$. Then for any $y \in \Mlin$, there exists $a>y$ such that $a$ realises $p(x)$ and $ga<a$. 
\item If for all $y \in \Mlin$, there exists $z<y$ such that $gz>z$. Suppose $p^{<} (x) \vdash \{ x<x' :x' \in X\}$, then for any $y \in \Mlin$, there exists $a<y$ such that $a$ realises $p(x)$ and $ga>a$. 
\item If for all $y \in \Mlin$, there exists $z<y$ such that $gz<z$. Suppose $p^{<} (x) \vdash \{ x<x' :x' \in X\}$, then for any $y \in \Mlin$, there exists $a<y$ such that $a$ realises $p(x)$ and $ga<a$. 
\end{enumerate}
\end{coro}
\begin{proof}
To prove (i), by the assumption on $g$, we can find $b > max\{ y, x' : x'\in X\}$ such that $gb>b$. Then $p^{<} (x)  \cup \{ b<x<gb \}$ is consistent. So, by the previous lemma, we have $a \in (b,gb)$ realising $p(x)$. We also have that $ga>gb>a$. We can prove (ii)-(iv) in the same way.
\end{proof}

\begin{prop}\label{oneorbitalcase1}
Let $g \in Aut(\Mlin)$ be such that for any $x\in \Mlin$, there exists $y<x$ and $z>x$ such that $gy>y$ and $gz>z$. Then there exist $h \in Aut(\Mlin)$ such that $gh^{-1}gh$ has a single $+$-orbital.
\end{prop}
\begin{proof}
List all elements of $\Mlin$ as $x_0,x_1,...$. We construct $h$ using the back-and-forth method and find a sequence $(a_n)_{n \in \mathbb{Z}}$ such that $gh^{-1}gha_i=a_{i+1}>a_i$ for all $i \in \mathbb{Z}$ and for any $x\in \Mlin$, there exist $m,n \in \mathbb{Z}$ such that $a_m<x<a_n$, i.e. the $+$-orbital containing $a_0$ is unbounded below and above. Hence, $gh^{-1}gh$ has a single $+$-orbital.

\begin{figure}[h]
\centering
\begin{tikzpicture}
  \draw (0,0) node (1a){}
++(0.5,0) node[draw,circle,fill=black,minimum size=4pt,inner sep=0pt] (1) {}
++(0,0) node[above] (1n)  {$a_{-n-1}$}
     ++(1.2,0) node[draw,circle,fill=black,minimum size=4pt,inner sep=0pt] (2){} 
++(0,0) node[above] (2n)  {$g^{-1}a_{-n}$}
   ++(0.8,0) node[draw,circle,fill=black,minimum size=4pt,inner sep=0pt] (3)  {}
++(0.2,0) node[above] (3n)  {$a_{-n}$}
  ++(0.5,0) node[draw,circle,fill=black,minimum size=2pt,inner sep=0pt] (7) {}
     ++(0,0) node[below] (3x)  {$x_n$}
   ++(0.5,0) node[draw,circle,fill=black,minimum size=4pt,inner sep=0pt] (4) {}
   ++(0,0) node[above] (4n)  {$a_{-1}$}
     ++(1.1,0) node[draw,circle,fill=black,minimum size=4pt,inner sep=0pt] (5) {}
     ++(0,0) node[above] (5n)  {$g^{-1}a_0$}
  ++(0.8,0) node[draw,circle,fill=black,minimum size=4pt,inner sep=0pt] (6) {}
     ++(0,0) node[above] (6n)  {$a_{0}$}
  ++(0.8,0) node[draw,circle,fill=black,minimum size=2pt,inner sep=0pt] (7) {}
     ++(0,0) node[below] (7n)  {$x_0$}
 ++(0.8,0) node[draw,circle,fill=black,minimum size=4pt,inner sep=0pt] (8)  {}
    ++(0,0) node[above] (8n)  {$c_{0}$}
 ++(1.2,0) node[draw,circle,fill=black,minimum size=4pt,inner sep=0pt] (9)  {}
    ++(0,0) node[above] (9n)  {$a_{1}$}
   ++(0,0) node[above] (9n)  {$=gc_0$}
 ++(1.8,0) node[draw,circle,fill=black,minimum size=4pt,inner sep=0pt] (10)  {}
    ++(0,0) node[above] (10n)  {$a_{n}$}
++(1.2,0) node[draw,circle,fill=black,minimum size=4pt,inner sep=0pt] (11)  {}
    ++(0,0) node[above] (11n)  {$c_{n}$}
 ++(0.8,0) node[draw,circle,fill=black,minimum size=4pt,inner sep=0pt] (11x)  {}
    ++(0,0) node[above] (11xn)  {$a_{n+1}$}
        ++(0.5,0) node (11a)  {}
++(0,-2) node (12a)  {}
++(-1,0) node[draw,circle,fill=black,minimum size=4pt,inner sep=0pt] (12){} 
++(0,0) node[below] (12n)  {$gb_n$}
++(-1,0) node[draw,circle,fill=black,minimum size=4pt,inner sep=0pt] (13){} 
++(0,0) node[below] (13n)  {$b_n$}
++(-2,0) node[draw,circle,fill=black,minimum size=4pt,inner sep=0pt] (14){} 
++(0,0) node[below] (14n)  {$gb_{0}$}
++(-1,0) node[draw,circle,fill=black,minimum size=4pt,inner sep=0pt] (15){} 
++(0,0) node[below] (15n)  {$b_0$}
++(-1.3,0) node[draw,circle,fill=black,minimum size=2pt,inner sep=0pt] (16){} 
++(0,0) node[below] (16n)  {$x_0$}
++(-1,0) node[draw,circle,fill=black,minimum size=4pt,inner sep=0pt] (17){} 
++(0,0) node[below] (17n)  {$b_{-1}$}
++(-1,0) node[draw,circle,fill=black,minimum size=4pt,inner sep=0pt] (18){} 
++(0,0) node[below] (18n)  {$g^{-1}b_{-1}$}
  ++(-1.2,0) node[draw,circle,fill=black,minimum size=2pt,inner sep=0pt] (7) {}
     ++(0,0) node[below] (18x)  {$x_n$}
++(-0.7,0) node[draw,circle,fill=black,minimum size=4pt,inner sep=0pt] (19){} 
++(0,0) node[below] (19n)  {$b_{-n-1}$}
++(-1.6,0) node[draw,circle,fill=black,minimum size=4pt,inner sep=0pt] (20){} 
++(0,0) node[below] (20n)  {$g^{-1}b_{-n-1}$}
++(-0.9,0) node (19a){};

\draw[->] (1) -- (20) node[draw=none,fill=none,font=\scriptsize,midway,right]{};
\draw[->] (2) -- (19) node[draw=none,fill=none,font=\scriptsize,midway,right]{};
\draw[->] (5) -- (17) node[draw=none,fill=none,font=\scriptsize,midway,right]{};
\draw[->] (4) -- (18) node[draw=none,fill=none,font=\scriptsize,midway,right]{};
\draw[->] (8) -- (14) node[draw=none,fill=none,font=\scriptsize,midway,right]{};
\draw[->] (6) -- (15) node[draw=none,fill=none,font=\scriptsize,midway,right]{};
\draw[->] (10) -- (13) node[draw=none,fill=none,font=\scriptsize,midway,right]{};
\draw[->] (11) -- (12) node[draw=none,fill=none,font=\scriptsize,midway,right]{};
\draw[-] (1a) -- (11a) node[draw=none,fill=none,font=\scriptsize,midway,right]{};
\draw[-] (19a) -- (12a) node[draw=none,fill=none,font=\scriptsize,midway,right]{};
\end{tikzpicture}
\end{figure}

We start with $\tilde{h}$ as the empty map. Choose $a_0<x_0<b_0$ such that $ga_0>a_0, gb_0>b_0$. Extend $\tilde{h}$ by sending $a_0$ to $b_0$. By the hypothesis and Corollary \ref{exofreal}, we can find $c_0 >x_0,a_0$ such that $c_0$ realises $\tilde{h}^{-1} \cdot tp(gb_0/b_0)$ and $gc_0 >c_0$ since $\tilde{h}^{-1}\cdot tp(gb_0/b_0) \vdash \{x>a_0 \}$. Extend $\tilde{h}$ by sending $c_0$ to $gb_0$. Let $a_1:=gc_0$. Then $g\tilde{h}^{-1}g\tilde{h} a_0=a_1>a_0$ since $gc_0>c_0>a_0$. 

Now since $\tilde{h}\cdot tp(g^{-1}a_0/a_0c_0) \vdash \{ x<b_0,gb_0\}$, by Corollary \ref{exofreal}, we can find $b_{-1}<x_0$ such that $b_{-1}$ realises $\tilde{h}\cdot tp(g^{-1}a_0/a_0c_0)$ and $gb_{-1}>b_{-1}$. Extend $\tilde{h}$ by sending $g^{-1}a_0$ to $b_{-1}$. Similarly, we can choose $a_{-1}$ realising $\tilde{h}^{-1} \cdot tp(g^{-1}b_{-1}/b_0b_{-1}gb_0)$ such that $ga_{-1}>a_{-1}$ and extend $\tilde{h}$ by sending $a_{-1}$ to $g^{-1}b_{-1}$. Then $g\tilde{h}^{-1}g\tilde{h}a_{-1}=a_0>g^{-1}a_0>a_{-1}$. 

Let $A_0=\{a_{-1},g^{-1}a_0,a_0,c_0\}$ and $B_0=\{g^{-1}b_{-1},b_{-1},b_0,gb_0\}$ be the domain and image of $\tilde{h}$. Then, $x_0$ is in both $(\min A_0, \max A_0)$ and $(\min B_0,\max B_0)$. Choose $y_0$ realising $\tilde{h}\cdot tp(x_0/A_0)$ and extend $\tilde{h}$ by sending $x_0$ to $y_0$. Choose $z_0$ realising $\tilde{h}^{-1}\cdot tp(x_0/y_0B_0)$ and extend $\tilde{h}$ by sending $z_0$ to $x_0$. Extend $A_0$,$B_0$ to include $x_0,z_0$ and $y_0,x_0$ respectively.

Suppose at the $n$-th step, we have a partial isomorphism, $\tilde{h}: A_{n-1} \rightarrow B_{n-1}$ satisfying 
\begin{enumerate}[(i)]
\item $A_{n-1} =\{a_{-n},g^{-1}a_{-n+1},...,a_{0},c_0,...,a_{n-1},c_{n-1},x_0,...,x_{n-1},z_0,...,z_{n-1}\}$, $B_{n-1}=\{g^{-1}b_{-n},b_{-n},....,b_{-1},b_0,...,b_{n-1},gb_{n-1},x_0,...,x_{n-1},y_0,...,y_{n-1}\} $,
\item $\tilde{h}$ maps $a_i$ to $b_i$, $c_i$ to $gb_i$, $x_i$ to $y_i$ and $z_i$ to $x_i$ for all $i=0,...,n-1$ and $\tilde{h}$ maps $a_j$ to $g^{-1}b_j$ and $g^{-1}a_{j+1}$ to $b_j$ for all $j=-1,...,-n$,
\item $g\tilde{h}^{-1}g$ $\tilde{h} a_{i}=a_{i+1}>a_{i}$ for all $i=-n,...,n-1$, and
\item  $\min A_{n-1}=a_{-n}$, $\max A_{n-1}=c_{n-1}$, $\min B_{n-1}=g^{-1}b_{-n}$, $\max B_{n-1}=gb_{n-1}$.
\end{enumerate}

Since $\tilde{h}\cdot tp(a_n/A_{n-1}) \vdash \{ x>b |b \in B_{n-1}\}$, by Corollary \ref{exofreal}, we can choose $b_n >x_n$ realising $\tilde{h} \cdot tp(a_n/A_{n-1})$ such that $gb_n>b_n$ and extend $\tilde{h}$ by sending $a_n$ to $b_n$. We can find $c_{n} >x_n$ realising $\tilde{h}^{-1}\cdot tp(gb_n/b_nB_{n-1})$ such that $gc_{n}>c_{n}$ and extend $\tilde{h}$ by sending $c_n$ to $gb_n$. Let $a_{n+1}:=gc_{n}$. Then $g\tilde{h}^{-1}g\tilde{h} a_n=a_{n+1}>a_n$. Similarly, since $\tilde{h}\cdot tp(g^{-1}a_{-n}/a_nc_nA_{n-1}) \vdash \{ x<b |b\in b_ngb_nB_{n-1}\}$, we can choose $b_{-n-1} <x_n$ realising $\tilde{h}\cdot tp(g^{-1}a_{-n}/a_n c_{n}A_{n-1})$ such that $gb_{-n-1}>b_{-n-1}$. Extend $\tilde{h}$ by sending $g^{-1}a_{-n}$ to $b_{-n-1}$. Choose $a_{-n-1} <x_n$ such that $a_{-n-1}$ realises $\tilde{h}^{-1}\cdot tp(g^{-1}b_{-n-1}/b_nb_{-n-1}gb_nB_{n-1})$ and $ga_{-n-1}>a_{-n-1}$. Extend $\tilde{h}$ by sending $a_{-n-1}$ to $g^{-1}b_{-n-1}$. Then $g\tilde{h}^{-1}g\tilde{h}a_{-n-1}=a_{-n}>a_{-n-1}$.

Let $A_n=A_{n-1}\cup \{a_{-n-1},g^{-1}a_{-n},a_n,c_n\}$ and $B_n=B_{n-1}\cup\{g^{-1}b_{-n-1},\\ b_{-n-1},b_n,gb_n\}$. Then $x_n$ is in both $( \min A_n, \max A_n)$ and $(\min B_n, \max B_n)$. Find $y_n$ realising $\tilde{h} \cdot tp(x_n/A_n)$ and extend $\tilde{h}$ by sending $x_n$ to $y_n$. Find $z_n$ realising $\tilde{h}^{-1}\cdot tp(x_n/y_nB_n)$. Extend $A_n$,$B_n$ to include $x_n,z_n$ and $y_n,x_n$ respectively. Then, $\min A_n=a_{-n-1}$, $\max A_n=c_n$, $\min B_n=g^{-1}b_{-n_1}$, $\max B_n=gb_{n}$. Hence $\tilde{h}: A_n\rightarrow B_n$ satisfies the hypothesis (i)-(iv).

Let $h$ be the union of $\tilde{h}$ over all steps. Then $gh^{-1}gh$ maps $a_i$ to $a_{i+1}>a_i$ for all $i \in \mathbb{Z}$ and for any $x\in \Mlin$, there exist $n,m  \in \mathbb{Z}$ such that $x<x_n<c_n<a_{n+1}$ and $x>x_m>a_{-m}$. Therefore, $gh^{-1}gh$ has a single $+$-orbital.
\end{proof}

We prove the following two propositions using a similar approach as the previous one.

\begin{prop}\label{2orbitals}
Let $g \in Aut(\Mlin)$ be such that for all $x \in \Mlin$, there exists $y>x$ and $z<x$ such that $gy>y$ and $gz<z$. Then there exists $h \in Aut(\Mlin)$ such that $gh^{-1}gh$ has a $+$-orbital unbounded above and a $-$-orbital unbounded below.
\end{prop}
\begin{proof}
List all elements of $\Mlin$ as $x_0,x_1,...$. We construct $h$ using the back-and-forth method and find sequences $(a_n)_{n \in \mathbb{Z}}, (a'_n)_{n \in \mathbb{Z}}$ such that $gh^{-1}gha_i=a_{i+1}>a_i$, $gh^{-1}gha'_i=a'_{i+1}<a'_i$ for all $i \in \mathbb{Z}$ and for any $x\in \Mlin$, there exists $m,n \in \mathbb{Z}$ such that $a'_m<x<a_n$. Then the $+$-orbital containing $a_0$ is unbounded above and the $-$-orbital containing $a'_0$ is unbounded below

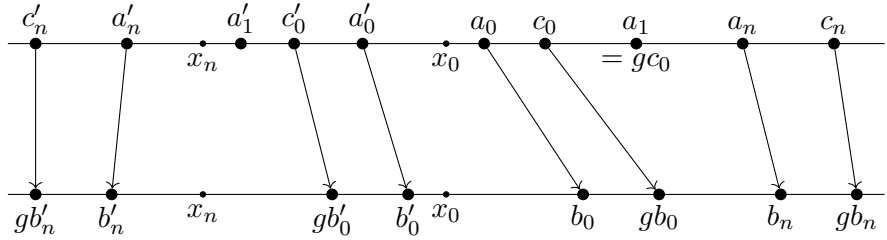
\begin{figure}[h]
\centering
\begin{tikzpicture}
  \draw (0,0) node (1a){}
++(0.5,0) node[draw,circle,fill=black,minimum size=4pt,inner sep=0pt] (1) {}
++(0,0) node[above] (1n)  {$c'_{n}$}
     ++(1.2,0) node[draw,circle,fill=black,minimum size=4pt,inner sep=0pt] (2){} 
++(0,0) node[above] (2n)  {$a'_{n}$}
  ++(1.0,0) node[draw,circle,fill=black,minimum size=2pt,inner sep=0pt] (7) {}
     ++(0,0) node[below] (3x)  {$x_n$}
   ++(0.5,0) node[draw,circle,fill=black,minimum size=4pt,inner sep=0pt] (3)  {}
++(0,0) node[above] (3n)  {$a'_{1}$}
   ++(0.7,0) node[draw,circle,fill=black,minimum size=4pt,inner sep=0pt] (4) {}
   ++(0,0) node[above] (4n)  {$c'_{0}$}
     ++(0.9,0) node[draw,circle,fill=black,minimum size=4pt,inner sep=0pt] (5) {}
     ++(0,0) node[above] (5n)  {$a'_0$}
  ++(1.1,0) node[draw,circle,fill=black,minimum size=2pt,inner sep=0pt] (7) {}
     ++(0,0) node[below] (7n)  {$x_0$}
  ++(0.5,0) node[draw,circle,fill=black,minimum size=4pt,inner sep=0pt] (6) {}
     ++(0,0) node[above] (6n)  {$a_{0}$}
 ++(0.8,0) node[draw,circle,fill=black,minimum size=4pt,inner sep=0pt] (8)  {}
    ++(0,0) node[above] (8n)  {$c_{0}$}
 ++(1.2,0) node[draw,circle,fill=black,minimum size=4pt,inner sep=0pt] (9)  {}
    ++(0,0) node[above] (9n)  {$a_{1}$}
    ++(0,0) node[below] {$=gc_0$}
 ++(1.4,0) node[draw,circle,fill=black,minimum size=4pt,inner sep=0pt] (10)  {}
    ++(0,0) node[above] (10n)  {$a_{n}$}
++(1.2,0) node[draw,circle,fill=black,minimum size=4pt,inner sep=0pt] (11)  {}
    ++(0,0) node[above] (11n)  {$c_{n}$}
        ++(0.8,0) node (11a)  {}
++(0,-2) node (12a)  {}
++(-0.5,0) node[draw,circle,fill=black,minimum size=4pt,inner sep=0pt] (12){} 
++(0,0) node[below] (12n)  {$gb_n$}
++(-1,0) node[draw,circle,fill=black,minimum size=4pt,inner sep=0pt] (13){} 
++(0,0) node[below] (13n)  {$b_n$}
++(-1.6,0) node[draw,circle,fill=black,minimum size=4pt,inner sep=0pt] (14){} 
++(0,0) node[below] (14n)  {$gb_{0}$}
++(-1,0) node[draw,circle,fill=black,minimum size=4pt,inner sep=0pt] (15){} 
++(0,0) node[below] (15n)  {$b_0$}
++(-1.8,0) node[draw,circle,fill=black,minimum size=2pt,inner sep=0pt] (16){} 
++(0,0) node[below] (16n)  {$x_0$}
++(-0.5,0) node[draw,circle,fill=black,minimum size=4pt,inner sep=0pt] (17){} 
++(0,0) node[below] (17n)  {$b'_{0}$}
++(-1,0) node[draw,circle,fill=black,minimum size=4pt,inner sep=0pt] (18){} 
++(0,0) node[below] (18n)  {$gb'_{0}$}
  ++(-1.7,0) node[draw,circle,fill=black,minimum size=2pt,inner sep=0pt] (7) {}
     ++(0,0) node[below] (18x)  {$x_n$}
++(-1.2,0) node[draw,circle,fill=black,minimum size=4pt,inner sep=0pt] (19){} 
++(0,0) node[below] (19n)  {$b'_{n}$}
++(-1,0) node[draw,circle,fill=black,minimum size=4pt,inner sep=0pt] (20){} 
++(0,0) node[below] (20n)  {$gb'_{n}$}
++(-0.5,0) node (19a){};

\draw[->] (1) -- (20) node[draw=none,fill=none,font=\scriptsize,midway,right]{};
\draw[->] (2) -- (19) node[draw=none,fill=none,font=\scriptsize,midway,right]{};
\draw[->] (5) -- (17) node[draw=none,fill=none,font=\scriptsize,midway,right]{};
\draw[->] (4) -- (18) node[draw=none,fill=none,font=\scriptsize,midway,right]{};
\draw[->] (8) -- (14) node[draw=none,fill=none,font=\scriptsize,midway,right]{};
\draw[->] (6) -- (15) node[draw=none,fill=none,font=\scriptsize,midway,right]{};
\draw[->] (10) -- (13) node[draw=none,fill=none,font=\scriptsize,midway,right]{};
\draw[->] (11) -- (12) node[draw=none,fill=none,font=\scriptsize,midway,right]{};
\draw[-] (1a) -- (11a) node[draw=none,fill=none,font=\scriptsize,midway,right]{};
\draw[-] (19a) -- (12a) node[draw=none,fill=none,font=\scriptsize,midway,right]{};
\end{tikzpicture}
\end{figure}

We start with $\tilde{h}$ as the empty map. Choose $a_0,b_0>x_0$ such that $gb_0>b_0$. Extend $\tilde{h}$ by sending $a_0$ to $b_0$. By the hypothesis and Corollary \ref{exofreal}, we can find $c_0 >x_0$ such that $c_0$ realises $\tilde{h}  \cdot tp(gb_0/b_0)$ and $gc_0 >c_0$ since $\tilde{h}^{-1}\cdot tp(gb_0/b_0) \vdash \{x>a_0 \}$. Extend $\tilde{h}$ by sending $c_0$ to $gb_0$. Let $a_1:=gc_0$. Then $g\tilde{h}^{-1}g\tilde{h} a_0=a_1>a_0$ since $gc_0>c_0>a_0$. 

Similarly, by Corollary \ref{exofreal}, we can find $a'_0,b'_0<x_0$ such that $b'_{0}$ realises $\tilde{h}\cdot tp(a'_0/a_0c_0)$ and $gb'_{0}<b'_{0}$. Extend $\tilde{h}$ by sending $a'_0$ to $b'_{0}$. We can choose $c'_{0}$ realising $\tilde{h}^{-1}\cdot tp(gb'_{0}/b_0b'_{0}gb_0)$ such that $gc'_{0}<c'_{0}$ since $\tilde{h}^{-1}\cdot tp(gb'_{0}/b_0b'_{0}gb_0) \vdash \{ x<a'_0,a_0,c_0 \}$. Extend $\tilde{h}$ by sending $c'_{0}$ to $gb'_{0}$. Let $a'_1:=gc'_0$. Then $g\tilde{h}^{-1}g\tilde{h}a'_{0}=a'_1<a'_0$. 

Let $A_0=\{c'_0,a'_0,a_0,a_0,c_0\}$ and $B_0=\{gb'_{0},b'_{0},b_0,gb_0\}$ be the domain and image of $\tilde{h}$. Then, $x_0$ is in both $(\min A_0, \max A_0)$ and $(\min B_0,\max B_0)$. Choose $y_0$ realising $\tilde{h}\cdot tp(x_0/A_0)$ and extend $\tilde{h}$ by sending $x_0$ to $y_0$. Choose $z_0$ realising $\tilde{h}^{-1}\cdot tp(x_0/y_0B_0)$ and extend $\tilde{h}$ by sending $z_0$ to $x_0$. Extend $A_0$,$B_0$ to include $x_0,z_0$ and $y_0,x_0$ respectively.

Now suppose at the $n$-th step, we have a partial isomorphism of $\M$, $\tilde{h}: A_{n-1} \rightarrow B_{n-1}$ satisfying
\begin{enumerate}[(i)]
\item $A_{n-1} =\{c'_{n-1},a'_{n-1},...,a'_0,a_{0},...,a_{n-1},c_{n-1},x_0,...,x_{n-1},$ $z_0,...,z_{n-1}\}$, \\ $B_{n-1}=\{gb'_{n-1},b'_{n-1},....,b'_0,b_0,...,b_{n-1},gb_{n-1},x_0,...,x_{n-1},$ $y_0,...,y_{n-1}\} $,
\item $\tilde{h}$ maps $a_i$ to $b_i$, $c_i$ to $gb_i$, $a'_i$ to $b'_i$, $c'_i$ to $gb'_i$, $x_i$ to $y_i$ and $z_i$ to $x_i$ for all $i=0,...,n-1$,
\item $g\tilde{h}^{-1}g\tilde{h} a_{i}=a_{i+1}>a_{i}$ and $g\tilde{h}^{-1}g\tilde{h} a'_{i}=a'_{i+1}<a'_{i}$ for all $i=0,...,n-1$, and
\item  $\min A_{n-1}=c'_{n-1}$, $\max A_{n-1}=c_{n-1}$, $\min B_{n-1}=gb'_{n-1}$, $\max B_{n-1}=gb_{n-1}$.
\end{enumerate}

Since $\tilde{h}\cdot tp(a_n/A_{n-1}) \vdash \{ x>b | b \in B_{n-1}\}$, we can choose $b_n >x_n$ realising $\tilde{h}\cdot tp(a_n/A_{n-1})$ such that $gb_n>b_n$ and extend $\tilde{h}$ by sending $a_n$ to $b_n$. We can find $c_{n} >x_n$ realising $\tilde{h}^{-1}\cdot tp(gb_n/b_nB_{n-1})$ such that $gc_{n}>c_{n}$ and extend $\tilde{h}$ by sending $c_n$ to $gb_n$. Let $a_{n+1}:=gc_{n}$. Then $g\tilde{h}^{-1}g\tilde{h} a_n=a_{n+1}>a_n$. Similarly, since $\tilde{h}\cdot tp(a'_{n}/a_nc_nA_{n-1}) \vdash \{ x<b |b\in b_ngb_nB_{n-1}\}$, we can choose $b'_{n} <x_n$ realising $\tilde{h}\cdot tp(a'_n/a_n c_{n}A_{n-1})$ such that $gb'_{n}<b'_{n}$. Extend $\tilde{h}$ by sending $a'_{n}$ to $b'_{n}$. Choose $c'_{n} <x_n$ such that $a'_{n}$ realises $\tilde{h}^{-1}\cdot tp(gb'_{n}/b_nb'_{n}gb_nB_{n-1})$ and $gc'_{n}<c'_{n}$. Extend $\tilde{h}$ by sending $c'_n$ to $gb'_{n}$. Let $a'_{n+1} =gc'_n$. Then $g\tilde{h}^{-1}g\tilde{h}a'_{n-1}=a'_{-n}<a'_{n-1}$.

Let $A_n=A_{n-1} \cup\{a'_{n},c'_n,a_n,c_n\}$ and $B_n=B_{n-1} \cup \{b_{-n-1},g^{-1}b_{-n-1},b_n,\\ gb_n\}$. Then $x_n$ is in both $( \min A_n, \max A_n)$ and $(\min B_n,\max B_n)$. Find $y_n$ realising $\tilde{h}\cdot tp(x_n/A_n)$ and extend $\tilde{h}$ by sending $x_n$ to $y_n$. Find $z_n$ realising $\tilde{h}^{-1}\cdot tp(x_n/y_nB_n)$. Extend $A_n$,$B_n$ to include $x_n,z_n$ and $y_n,x_n$ respectively. Then, $\min A_n=c'_n$, $\max A_n=c_n$, $\min B_n=gb'_{n}$, $\max A_n=gb_{n}$. Hence $\tilde{h}: A_n\rightarrow B_n$ satisfies the hypothesis (i)-(iv).

Let $h$ be the union of $\tilde{h}$ over all steps. Then $gh^{-1}gha_i=a_{i+1}>a_i$, $gh^{-1}gha'_i=a'_{i+1}<a'_i$ for all $i \in \mathbb{Z}$ and for any $x\in \Mlin$, there exists $m,n \in \mathbb{Z}$ such that $a'_{m+1}<c'_m<x_m<x<x_n<c_n<a_{n+1}$.
\end{proof}

\begin{prop}\label{oneorbitalcase2}
Let $g \in Aut(\Mlin)$ be such that $g$ has a $+$-orbital unbounded above and a $-$-orbital unbounded below. Then there exists $h \in Aut(\Mlin)$ such that $[g,h]$ has a single $+$-orbital.
\end{prop}
\begin{proof}
Let $y \in \Mlin$ be an element of the $+$-orbital unbounded above and let $z\in \Mlin$ be an element of the $-$-orbital unbounded below. Then, we have that $gx>x$ for all $x>y$ and $gx<x$ for all $x<z$. List all elements of $\M$ as $x_0,x_1,...$. We construct $h$ using the back-and-forth method and find a sequence $(a_n)_{n\in \mathbb{Z}}$ such that $[g,h]a_i=a_{i+1}>a_i$ for all $i \in \mathbb{Z}$ and for any $x\in \Mlin$, there exist $m,n \in \mathbb{Z}$ such that $a_m<x<a_n$. Then $[g,h]$ has a single $+$-orbital. We start with $\tilde{h}$ as the empty map. 

\begin{figure}[h]
\centering
\begin{tikzpicture}
  \draw (0,0) node (1a) {}
     ++(0.5,0) node[draw,circle,fill=black,minimum size=4pt,inner sep=0pt] (1){} 
   ++(0,0) node[above] (1n)  {$ga_{-n}$}
     ++(1.0,0) node[draw,circle,fill=black,minimum size=4pt,inner sep=0pt] (2){} 
   ++(0,0) node[above] (2n)  {$a_{-n-1}$}
   ++(0.8,0) node[draw,circle,fill=black,minimum size=4pt,inner sep=0pt] (3)  {}
++(0,0) node[above] (3n)  {$a_{-n}$}
   ++(1.5,0) node[draw,circle,fill=black,minimum size=4pt,inner sep=0pt] (4) {}
   ++(0,0) node[above] (4n)  {$ga_{0}$}
     ++(0.8,0) node[draw,circle,fill=black,minimum size=4pt,inner sep=0pt] (5) {}
     ++(0,0) node[above] (5n)  {$a_{-1}$}
     ++(0.8,0) node[draw,circle,fill=black,minimum size=4pt,inner sep=0pt] (20) {}
     ++(0,0) node[above] (20n)  {$a_{0}$}
  ++(0.8,0.5) node (6n)  {}
  ++(0,-3) node (6m) {}
  ++(1,0) node (7m) {}
 ++(0,3) node (7n)  {}
  ++(-0.2,-.5) node[above] (7a)  {$y$}
    ++(-1,0) node[above] (6a)  {$z$}
 ++(2.2,0) node[draw,circle,fill=black,minimum size=4pt,inner sep=0pt] (8)  {}
    ++(0,0) node[above] (8n)  {$a_{1}$}
 ++(0,0) node[below]  {$=g^{-1}c_0$}
 ++(1.2,0) node[draw,circle,fill=black,minimum size=4pt,inner sep=0pt] (9)  {}
    ++(0,0) node[above] (9n)  {$c_{0}$}
 ++(1.5,0) node[draw,circle,fill=black,minimum size=4pt,inner sep=0pt] (10)  {}
    ++(0,0) node[above] (10n)  {$a_{n}$}
 ++(0.8,0) node[draw,circle,fill=black,minimum size=4pt,inner sep=0pt] (11)  {}
    ++(0,0) node[above] (11n)  {$a_{n+1}$}
 ++(0.8,0) node[draw,circle,fill=black,minimum size=4pt,inner sep=0pt] (21)  {}
    ++(0,0) node[above] (21n)  {$c_n$}
        ++(0.5,0) node (11a)  {}
++(0,-2) node (12a)  {}
++(-0.8,0) node[draw,circle,fill=black,minimum size=4pt,inner sep=0pt] (12){} 
++(0,0) node[below] (12n)  {$gb_n$}
++(-1.2,0) node[draw,circle,fill=black,minimum size=4pt,inner sep=0pt] (13){} 
++(0,0) node[below] (13n)  {$b_n$}
++(-1.9,0) node[draw,circle,fill=black,minimum size=4pt,inner sep=0pt] (14){} 
++(0,0) node[below] (14n)  {$gb_{0}$}
++(-1.0,0) node[draw,circle,fill=black,minimum size=4pt,inner sep=0pt] (15){} 
++(0,0) node[below] (15n)  {$b_{0}$}
++(-2.5,0) node[draw,circle,fill=black,minimum size=4pt,inner sep=0pt] (18){} 
++(0,0) node[below] (18n)  {$g^{-1}b_{-1}$}
++(-1.5,0) node[draw,circle,fill=black,minimum size=4pt,inner sep=0pt] (19){} 
++(0,0) node[below] (19n)  {$b_{-1}$}
++(-1.8,0) node[draw,circle,fill=black,minimum size=4pt,inner sep=0pt] (22){} 
++(0,0) node[below] (22n)  {$g^{-1}b_{-n-1}$}
++(-1.4,0) node[draw,circle,fill=black,minimum size=4pt,inner sep=0pt] (23){} 
++(0,0) node[below] (22n)  {$b_{-n-1}$}
++(-0.9,0) node (19a){};

\draw[->] (1) -- (23) node[draw=none,fill=none,font=\scriptsize,midway,right]{};
\draw[->] (2) -- (22) node[draw=none,fill=none,font=\scriptsize,midway,right]{};
\draw[->] (20) -- (15) node[draw=none,fill=none,font=\scriptsize,midway,right]{};
\draw[->] (4) -- (19) node[draw=none,fill=none,font=\scriptsize,midway,right]{};
\draw[->] (5) -- (18) node[draw=none,fill=none,font=\scriptsize,midway,right]{};
\draw[->] (10) -- (13) node[draw=none,fill=none,font=\scriptsize,midway,right]{};
\draw[->] (9) -- (14) node[draw=none,fill=none,font=\scriptsize,midway,right]{};
\draw[->] (21) -- (12) node[draw=none,fill=none,font=\scriptsize,midway,right]{};
\draw[-] (1a) -- (11a) node[draw=none,fill=none,font=\scriptsize,midway,right]{};
\draw[-] (19a) -- (12a) node[draw=none,fill=none,font=\scriptsize,midway,right]{};
\draw[-] (6n) -- (6m) node[draw=none,fill=none]{};
\draw[-] (7n) -- (7m) node[draw=none,fill=none]{};
\end{tikzpicture}
\end{figure}

We can choose $a_0< z,x_0$ and $b_0>x_0,y$. Then $ga_0<a_0$ and $gb_0>b_0$. Extend $\tilde{h}$ by sending $a_0$ to $b_0$. By Corollary \ref{exofreal}, we can find $c_0 >x_0,gy$ realising $\tilde{h}^{-1}\cdot tp(gb_0/b_0)$ since $\tilde{h}^{-1}\cdot tp(gb_0/b_0) \vdash \{x>a_0 \}$. Extend $\tilde{h}$ by sending $c_0$ to $gb_0$. Let $a_1:=g^{-1}c_0$. Then $g^{-1}c_0=[g,\tilde{h}] a_0=a_1>y>z>a_0$.

Similarly, we can choose $b_{-1}<gz,x_0$ such that $b_{-1}$ realises $\tilde{h}\cdot tp(ga_0/c_0a_0)$. Then $g^{-1}b_{-1} < z$. Extend $\tilde{h}$ by sending $ga_0$ to $b_{-1}$. Choose $a_{-1}$ realising $\tilde{h}^{-1}\cdot  tp(g^{-1}b_{-1}/b_{-1}b_0gb_0)$ and extend $\tilde{h}$ by sending $a_{-1}$ to $g^{-1}b_{-1}$. Then $a_{-1}<a_0$ since $g^{-1}b_{-1}<z<y<b_0$. Then we have $[g,\tilde{h}] a_{-1}=a_0>a_{-1}$.

Let $A_0=\{ga_0,a_{-1},a_0,c_0\}$ and $B_0=\{b_{-1},g^{-1}b_{-1},b_0,gb_0\}$ be the domain and image of $\tilde{h}$. Then $x_0$ is in $(\min A_0, \max A_0)\cap(\min B_0,\max B_0)$. Choose $y_0$ realising $\tilde{h}\cdot tp(x_0/A_0)$ and extend $\tilde{h}$ by sending $x_0$ to $y_0$. Choose $z_0$ realising $\tilde{h}^{-1}\cdot tp(x_0/y_0B_0)$ and extend $\tilde{h}$ by sending $z_0$ to $x_0$. Extend $A_0$,$B_0$ to include $x_0,z_0$ and $y_0,x_0$ respectively.

Now suppose at the $n$-th step, we have a partial isomorphism of $\M$, $\tilde{h}: A_{n-1} \rightarrow B_{n-1}$ satisfying 
\begin{enumerate}[(i)]
\item $A_{n-1}= \{ga_{-n+1},a_{-n},...a_{-1},a_0,...,a_{n-1},c_{n-1},x_0,...,x_{n-1},z_0,...,z_{n-1}\}$ with $\min A_{n-1} =ga_{-n+1}<a_{-n}<\cdots<ga_0<a_{-1}<a_0<a_1<c_0<\cdots<c_{n-2}<a_{n}<c_{n-1}=\max A_{n-1}$,
\item $B_{n-1}=\{b_{-n},g^{-1}b_{-n},....,b_0,gb_0,...,b_{n-1},gb_{n-1},x_0,...,x_{n-1},y_0,...,y_{n-1}\}$ with $\min B_{n-1}=b_{-n}<g^{-1}b_{-n}<\cdots<b_{-1}<g^{-1}b_{-1}<b_0<gb_0<\cdots<b_{n-1}<gb_{n-1}=\max B_{n-1}$,
\item $\tilde{h}$ maps $a_i$ to $b_i$, $c_i$ to $gb_i$, $x_i$ to $y_i$ and $z_i$ to $x_i$ for all $i=0,...,n-1$ and $\tilde{h}$ maps $a_{-j}$ to $g^{-1}b_{-j}$ and $ga_{-j+1}$ to $b_{-j}$ for all $j=1,...,n$,
\item $[g,\tilde{h}] a_{i}=a_{i+1}>a_{i}$ for all $i=-n,...,n-1$, and
\item $c_i>g^ia_1$, $b_{i}>g^{\lfloor\frac{i}{2}\rfloor}b_0$  for all $i=1,...,n-1$ and $a_{-j}<g^{\lfloor\frac{j}{2}\rfloor} a_0$, $b_{-j}<g^{j-1}b_{-1}$ for all $j=2,...,n$.
\end{enumerate}

We can choose $b_n$ realising $\tilde{h}\cdot tp(a_n/A_{n-1})$ and extend $\tilde{h}$ by sending $a_n$ to $b_n$. Then $gb_n>b_n$. Since by the inductive hypothesis, $a_n>c_{n-2}$ and $\tilde{h}$ maps $c_{n-2}$ to $gb_{n-2}$, we have $b_n>gb_{n-2}$. By the inductive hypothesis (v), $b_{n-2}>g^{\lfloor\frac{n-2}{2}\rfloor}b_0$. Hence, we have $b_n>gb_{n-2}>g^{\lfloor\frac{n}{2}\rfloor}b_0$.

We can find $c_n>x_n, g^na_1, gc_{n-1}$ realising $\tilde{h}^{-1}\cdot tp(gb_n/b_nB_{n-1})$ and extend $\tilde{h}$ by sending $c_n$ to $gb_n$. Let $a_{n+1}:=g^{-1}c_n$. Then $a_{n+1}=[g,\tilde{h}] a_n>[g,\tilde{h}]a_{n-1}=a_n$ since $a_n >a_{n-1}$. We also have $a_{n+1}=g^{-1}c_n>c_{n-1}$ since $c_n>gc_{n-1}$.

Choose $b_{-n-1}<x_n, g^{n}b_{-1}, gb_{-n}$ realising $\tilde{h}\cdot tp(ga_{-n}/a_nc_nA_{n-1})$. Extend $\tilde{h}$ by sending $ga_{-n}$ to $b_{-n-1}$. Choose $a_{-n-1}$ realising $\tilde{h}^{-1}\cdot tp(g^{-1}b_{-n-1}/b_nb_{-n-1}$\\$gb_nB_{n-1})$ and extend $\tilde{h}$ by sending $a_{-n-1}$ to $g^{-1}b_{-n-1}$. Then $a_{-n-1}=[\tilde{h},g]a_{-n} <[\tilde{h},g]a_{-n+1} =a_{-n}$ since $a_{-n}<a_{-n+1}$. Since $b_{-n-1}< gb_{-n}$ and $\tilde{h}$ maps $ga_{-n+1}$ to $b_{-n}$ and $a_{-n-1}$ to $g^{-1}b_{-n-1}$, we have $a_{-n-1}<ga_{-n+1}$. By the inductive hypothesis (v), we have $ga_{-n+1}<g^{\lfloor\frac{n-1}{2}\rfloor} a_0$. Hence, $a_{-n-1}<ga_{-n+1}<g^{\lfloor\frac{n+1}{2}\rfloor} a_0$.

Now let $A_n=A_{n-1}\cup \{a_n,c_n,ga_{-n},a_{-n-1}\}$ and $B_n=B_{n-1}\cup \{b_{-n-1},$\\$g^{-1}b_{-n-1},b_n,gb_n \}$. By rearranging the list $x_n,x_{n+1},...$, we may assume $x_n$ is in $( \min A_n, \max A_n)\cap(\min B_n,\max B_n)$. Find $z_n$ realising $\tilde{h}^{-1}\cdot  tp(x_n/y_nB)$. Extend $A_n$,$B_n$ to include $x_n,z_n$ and $y_n,x_n$ respectively. Then, $\min B_n=b_{-n-1}$, $\max B_n=gb_{n}$, $\min A_n=ga_{-n}$, $\max A_n=c_{n}$. Hence, $\tilde{h}: A_n\rightarrow B_n$ satisfies the hypothesis (i)-(v).


Let $h$ be the union of $\tilde{h}$ over all steps. Since $g$ has a $+$-orbital unbounded above and a $-$-orbital unbounded below, by Remark \ref{orbitalrmk}, we know that $(g^ia_1)_{i\in \mathbb{Z}}, (g^ib_0)_{i\in \mathbb{Z}}$ are unbounded above and $(g^ia_0)_{i\in \mathbb{Z}}, (g^ib_{-1})_{i\in \mathbb{Z}}$ are unbounded below. Hence, for any $x\in \Mlin$, we can find $m \in \mathbb{Z}$ such that $c_m>g^ma_1>x$, $b_m>g^{\lfloor \frac{m}{2}\rfloor} b_0>x$, $a_{-m}<g^{\lfloor \frac{m}{2}\rfloor}a_0<x$ and $b_{-m}<g^{m-1}b_{-1}<x$. In other word, $\bigcup_{n \in \mathbb{N}}( \min A_n, \max A_n)\cap(\min B_n,\max B_n)=\Mlin$. Thus, $h$ is bijective and $h \in Aut(\Mlin)$.

Then we have that $[g,h]$ maps $a_i$ to $a_{i+1}>a_i$ for all $i \in \mathbb{Z}$ and for any $x\in \Mlin$, there exist $n,m  \in \mathbb{Z}$ such that $x<x_n<c_n<a_{n+1}$ and $x>x_m>a_{-m}$. Therefore, $[g,h]$ has a single $+$-orbital.
\end{proof}

\begin{thm}\label{oneorbitalgeneral}
Let $\M$ be a countable homogeneous $\lan$-structure and $\Mlin$ be its linear order expansion. Let $g$ be an unbounded automorphism of $Aut(\Mlin)$. Then there exists a product of conjugates of $g$ and $g^{-1}$ that has a single $+$-orbital. Similarly, there exists a product of conjugates of $g$ and $g^{-1}$ that has a single $-$-orbital.
\end{thm}
\begin{proof}
Suppose $g$ is an automorphism such that for any $x\in \Mlin$, there exists $y<x$ and $z>x$ such that $gy>y$ and $gz>z$. Then, by Proposition \ref{oneorbitalcase1}, there exist $h \in Aut(\Mlin)$ such that $ghgh^{-1}$ has a single $+$-orbital. Suppose $g$ is an automorphism such that for any $x\in \Mlin$, there exists $y<x$ and $z>x$ such that $gy<y$ and $gz<z$. Then, by applying Proposition \ref{oneorbitalcase1} on $g^{-1}$, we can find $h \in Aut(\Mlin)$ such that $g^{-1}hg^{-1}h^{-1}$ has a single $+$-orbital.

Suppose for all $x \in \Mlin$, there exists $y>x$ and $z<x$ such that $gy>y$ and $gz<z$. Then, by Proposition \ref{2orbitals}, we can find $h \in Aut(\Mlin)$ such that $gh^{-1}gh$ has a $+$-orbital unbounded above and a $-$-orbital unbounded below. Then by the previous proposition, we can find $k \in Aut(\Mlin)$ such that $[gh^{-1}gh,k]$ has a single $+$-orbital. Suppose for all $x \in \Mlin$, there exists $y>x$ and $z<x$ such that $gy<y$ and $gz>z$. Then, we apply the same argument on $g^{-1}$ and find $h,k \in Aut(\Mlin)$ such that $[g^{-1}h^{-1}g^{-1}h,k]$ has a single $+$-orbital.

Since $g$ is unbounded, these are the only cases we need to consider. Hence for any $g\in Aut(\Mlin)$, there exists a product of conjugates of $g$ and $g^{-1}$ that has a single $+$-orbital. We can take the inverse of the product to obtain a product of conjugates of $g$ and $g^{-1}$ that has a single $-$-orbital. 
\end{proof}

\begin{lem}\label{nofixedset}
$\M$ be a non-trivial countable relational homogeneous $\lan$-structure such that $Aut(\M)$ $\neq Sym(\M)$. Let $\Mlin$ be its linear order expansion. Then for any non-trivial $g\in Aut(\Mlin)$, $g$ does not have an open interval $I=(x,y)$ such that $ga=a$ for all $a\in I$.
\end{lem}
\begin{proof}
Since $\M$ is non-trivial, there is a non-trivial $n$-ary relation $R\in \lan$. Suppose there exists such an open interval $I$. Then for any $b \neq c \in \Mlin$, by the Extension Property, there exists $\bar{a} \in I^{n-1}$ such that $R(\bar{a},b)$ and $\neg R(\bar{a},c)$. Since $g\bar{a}=\bar{a}$, we have $gb \neq c$. Therefore, $gb=b$ for any $b\in \Mlin$, contradicting the non-triviality of $g$.
\end{proof}

\begin{rmk}\label{nontrivialunbounded}
Hence, by the previous lemma, if $\Mlin$ is not the dense linear order, then any non-trivial automorphism of $\Mlin$ is unbounded.
\end{rmk}

\begin{coro}
Let $\M$ be a non-trivial countable homogeneous $\lan$-structure and $\Mlin$ be its linear order expansion. Let $g \in Aut(\Mlin)$ be non-trivial. Then there exists a product of conjugates of $g$ and $g^{-1}$ that has a single $+$-orbital. Similarly, there exists a product of conjugates of $g$ and $g^{-1}$ that has a single $-$-orbital.
\end{coro}

For an unbounded automorphism $g \in Aut(\Mlin)$, we have constructed $h\in Aut(\Mlin)$ as a product of conjugates of $g$ and $g^{-1}$ such that $ha>a$ for all $a \in \Mlin$. We now prove that such $h$ has some special property. We first define the following notion. 

\begin{dfn}
Let $A$ be a finite subset of $\Mlin$ and $b\in \Mlin \setminus A$. We definte the \emph{upper constraint} of $b$ with respect to $A$ to be $\min_{a\in A} \{a>b\}$ and the \emph{lower constraint} of $b$ with respect to $A$ to be $max_{a\in A} \{a<b\}$.
\end{dfn}

\begin{lem}\label{move}
Let $\M$ be a countable homogeneous $\lan$-structure with a $\mathbf{SWIR}$ $\indst$ and $\Mlin$ be its linear order expansion. Let $p$ be an $n$-type over some finite set $X$ and $B$ be some finite subset of $\Mlin$.
\begin{enumerate}[(i)]
\item Let $g \in Aut(\Mlin)$ be such that $ga < a$ for all $a\in \Mlin$. Then there exists $\bar{a}$ realising $p$ such that $\bar{a} \ind _X B$ and if $a_i \leq ga_j$ for some $a_i, a_j \in \bar{a}$, then there exists $x \in X$ such that $a_i \leq x\leq ga_j$.
\item Let $g \in Aut(\Mlin)$ be such that $ga > a$ for all $a \in \Mlin$. Then there exists $\bar{a}$ realising $p$ such that $B \ind _X \bar{a}$ and if $a_i \geq ga_j$ for some $a_i, a_j \in \bar{a}$, then there exists $x \in X$ such that $a_i \geq x\geq ga_j$.
\end{enumerate}
\end{lem}
\begin{proof}
(i) Let $p$ be an $n$-type over some finite set $X$. We can write $p=$tp$(x_1,\\...,x_n/X)$ where $x_1>\cdots >x_n$. Since $\bar{x} \cap X=\emptyset$, we can always choose $\bar{a}$ realising $p$ such that $\bar{a}X \cap g\bar{a}=\emptyset$. Let $p_k=$tp$(x_1,...,x_k/X)$ for $k=1,...,n$. We find $a_k$ inductively. For the inductive base, by Existence, we can choose $a_1$ realising $p_1$ such that $a_1 \ind _X B$. Then $ga_1 <a_1$. Now assume there exists $\bar{a}^{k-1}=(a_1,...,a_{k-1})$ satisfying 
(i) $\bar{a}^{k-1}$ realises $p_{k-1}$, 
(ii) $\bar{a}^{k-1} \ind _X B$, and 
(iii) if $a_i < ga_j$ for some $1\leq i,j \leq k-1$, then there exists $x \in X$ such that $a_i <x<ga_j$. 

By Existence, we can find $a_{k}$ such that $a_{k}$ realises $p_k(\bar{a}^{k-1},x_k/X)$ and $a_k \ind _{\bar{a}^{k-1}X} Bg\bar{a}^{k-1}$. Then by Monotonicity, we have $a_k \ind _{\bar{a}^{k-1}X} B$ and by Transitivity on it and the inductive hypothesis (ii) $\bar{a}^{k-1} \ind _X B$, we have $\bar{a}^{k-1}a_k \ind _{X} B$. We also have $a_k \ind _{\bar{a}^{k-1}X} g\bar{a}^{k-1}$ by Monotonicity. Then for any $1\leq j \leq k-1$, if $a_k< ga_j$, then either there exists $a_i \in \bar{a}^{k-1}$ such that $a_k<a_i<ga_j$ or there exists $x \in X$ such that $a_k <x<ga_j$. In the former case, by the inductive hypothesis (iii), there exists $x \in X$ such that $a_k<a_i <x<ga_j$. Hence, in both case, we can find $x \in X$ such that $a_k <x<ga_j$. We also have that $ga_k<a_k<a_i$ for all $i=1,...,k-1$. Thus, we have that if $a_i < ga_j$ for some $1\leq i,j \leq k$, then there exists $x \in X$ such that $a_i <x<ga_j$. Therefore, we have found $\bar{a}^{k}=(a_1,...,a_{k})$ satisfying  the inductive hypothesis (i)-(iii) for $k$.

By induction, we can find a realisation of $p$ satisfying the required property. By a symmetric argument, we can prove part (ii). 
\end{proof}

\begin{rmk}\label{nointersection}
It follows from the proof that if the type $p$ is of the form $tp(\bar{x}/X)$ such that $X$ is finite and for all $x_i \in \bar{X}$, $x_i \notin X$, then we can find a realisation $\bar{a}$ such that $\bar{a} \cap g\bar{a}=\emptyset$.
\end{rmk}

\section{Linearly ordered free homogeneous structures}

In this section, we let $\M$ be a non-trivial free homogeneous structure and $\Mlin$ be its linear order expansion. We can define a $\mathbf{SWIR}$ $\indst$ on $\M$ by defining $A \indst _B C$ for finite $A,B,C \subseteq \M$, if for any $a \in A \setminus B, c\in C \setminus B$, $(a,c)$ is not related by any relation in $\lan$. Then, by Theorem \ref{swirfromswir}, we can find a $\mathbf{SWIR}$ $\ind$ on $\Mlin$ by defining $A \ind _B C$ if for any $a \in A \setminus B, c\in C \setminus B$, $(a,c)$ is not related by any relation in $\lan$ and if $a<c$, then there exists $b\in B$ such that $a<b<c$.

Now for any non-trivial $g \in Aut(\Mlin)$, we want to find a product of conjugates of $g$ and $g^{-1}$ that moves $L$-maximally and its inverse moves $R$-maximally. We do this by first proving the following lemma and employ results from the previous section.

\begin{lem}\label{deletesupport}
Let $\Mlin$ be the linear order expansion of a non-trivial free homogeneous structure. Let $A,B,B',C$ be finite substructures of $\Mlin$ such that $A \cap B =\emptyset$ and $C \cap B= \emptyset$. Suppose $A \ind _{BB'} C$ and for any $b \in B\setminus B', c \in C \setminus B'$, if $b<c$, then there exists $b' \in B$ such that $b<b'<c$. Then $A \ind_{B'}  C$. 
\end{lem}
\begin{proof}
For $a \in A/B'$ and $c\in C/B'$, we have $a \in A/BB'$ and $c\in C/BB'$ since $A \cap B =\emptyset$ and $C \cap B= \emptyset$. Then since $A\ind _{BB'} C$, $(a,c)$ is not related by any relation. If $a <c$, then since $A\ind _{BB'} C$, either there exists $b\in B \setminus B'$ such that $a<b<c$ or there exists $b'\in B'$ such that $a<b'<c$. In the former case, by the assumption, there exists $b'\in B'$ such that $a<b<b'<c$. Therefore, we have $A \ind _{B'} C$.
\end{proof}

\begin{thm}
Let $g \in Aut(\M)$ be such that $ga < a$ for all $a \in \Mlin$. Then there exists $k \in Aut(\Mlin)$ such that $[k,g]$ moves almost $R$-maximally and $[g,k]$ moves almost $L$-maximally.
\end{thm}
\begin{proof}
We use the back-and-forth method. List all types as $p_1,p_2,...$. Suppose at some stage, we have a partial isomorphism $\tilde{k}:A \rightarrow B$ such that $[\tilde{k},g]$ moves $p_1,...,p_{i-1}$ almost $R$-maximally and $[g,\tilde{k}]$ moves $p_1,...,p_{i-1}$ almost $L$-maximally. Suppose $p_i=tp(\bar{x}/X)$. We may assume $p_i$ is non-algebraic and $x_i \notin X$ for all $x_i \in \bar{x}$. We may also assume that $X \cup g(X) \subseteq A$, $g\tilde{k} X \subseteq B$ by extending $\tilde{k}$.

Step 1. We extend $\tilde{k}$ so that $[ \tilde{k},g]$ moves $p_i$ almost $R$-maximally. By Lemma \ref{move}, there exists $\bar{a}$ realising $p_i$ such that $\bar{a} \ind _X A$  and if $a_i < ga_j$ for some $a_i, a_j \in \bar{a}$, then there exists $x \in X$ such that $a_i <x<ga_j$. By Remark \ref{nointersection}, we also have $\bar{a} \cap g\bar{a} = \emptyset$. By Existence and the fact that $tp(\bar{a}/A)$ is non-algebraic, there exists $\bar{b}$ realising $\tilde{k} \cdot tp(\bar{a}/A)$ such that $\bar{b} \ind _B g^{-1}B$ and $\bar{b} \cap g\bar{b} =\emptyset$. Extend $\tilde{k}$ by sending $\bar{a}$ to $\bar{b}$. By acting $\tilde{k}$ on $\bar{a} \ind _X A$, we get $\bar{b} \ind _{\tilde{k} X} B$.

Again by Existence, there exists $\bar{c}$ realising $ \tilde{k}^{-1} \cdot  tp(g\bar{b}/\bar{b}B)$ such that $\bar{c} \ind _{\bar{a}A} g\bar{a}$. Extend $\tilde{k}$ by sending $\bar{c}$ to $g\bar{b}$. Then we have $\bar{c}\cap \bar{a}=\emptyset$ from $\bar{b} \cap g\bar{b} =\emptyset$. Since if $a_i < ga_j$ for some $a_i, a_j \in \bar{a}$, then there exists $x \in X$ such that $a_i <x<ga_j$, we have $\bar{c} \ind _{A} g\bar{a}$ by Lemma \ref{deletesupport}. 

By Transitivity on $\bar{b} \ind _{\tilde{k} X} B$ and $\bar{b} \ind _B g^{-1}B$, we have $\bar{b} \ind _{\tilde{k} X}g^{-1} B$. Acting by $\tilde{k}^{-1}g$ on it, we get $\tilde{k}^{-1}g\bar{b} \ind _{\tilde{k}^{-1}g\tilde{k} X} \tilde{k}^{-1}B$, which can be simplified to $\bar{c} \ind _{\tilde{k}^{-1}g \tilde{k} X} A$. Since $ \tilde{k}^{-1}g \tilde{k} X \subseteq A $, we can apply Transitivity on $\bar{c} \ind _{\tilde{k}^{-1}g \tilde{k} X} A$ and $\bar{c} \ind _A g\bar{a}$ to obtain $ \bar{c} \ind _{\tilde{k}^{-1}g\tilde{k} X} g\bar{a}$. Acting by $\tilde{k}^{-1}g^{-1}  \tilde{k}$ on it, we have $\bar{a} \ind _X [\tilde{k},g] \bar{a}$.

Step 2. We extend $\tilde{k}$ so that $[g, \tilde{k}]$ moves $p_i$ almost $L$-maximally. Since $g^{-1}a>a$ for all $a \in \Mlin$, by Lemma \ref{move}(ii), we can find $\bar{a}$ realising $p$ such that $g^{-1}(A) \ind _X \bar{a}$ and if $a_i > g^{-1}a_j$ for some $a_i, a_j \in \bar{a}$, then there exists $x \in X$ such that $a_i >x>g^{-1}a_j$. Hence, if $ga_i > a_j$ for some $a_i, a_j \in a$, then there exists $x \in X$ such that $ga_i >gx>a_j$. We again have $\bar{a} \cap g\bar{a} = \emptyset$ by Remark \ref{nointersection}. By Invariance, we have $A \ind_{g(X)} g\bar{a}$.

By Existence and the fact that $tp(\bar{a}/A)$ is non-algebraic, we can find $\bar{b}$ realising $\tilde{k}\cdot tp(\bar{a}/A)$ such that $\bar{b}\cap g\bar{b}=\emptyset$ and $c$ realising $\tilde{k} ^{-1}\cdot tp(g\bar{b}/\bar{b}B)$ such that $\bar{c} \ind_{\bar{a}A}g\bar{a}$. Extend $\tilde{k}$ by sending $\bar{a}$ to $\bar{b}$ and $\bar{c}$ to $g\bar{b}$. Then we have $\bar{c} \cap \bar{a}=\emptyset$.

Since $g(X) \subseteq A$, $\bar{c} \cap \bar{a}=\bar{a} \cap g\bar{a} = \emptyset$ and if $ga_i > a_j$ for some $a_i, a_j \in \bar{a}$, then there exists $x \in X$ such that $ga_i >gx>a_j$, we obtain $\bar{c} \ind_{A}g\bar{a}$ from $\bar{c} \ind_{\bar{a}A}g\bar{a}$ by Lemma \ref{deletesupport}. Then by Transitivity on $A \ind_{g(X)} g\bar{a}$ and $\bar{c} \ind_{A}g\bar{a}$, we have $\bar{c}\ind_{g(X)} g\bar{a}$. Therefore, acting by $g^{-1}$ on it, we obtain $[g,\tilde{k}] \bar{a} \ind _X \bar{a}$.

We can also make sure $X \subset B$ by extending $\tilde{k}$. Let $k$ be the union of all $\tilde{k}$ over all types. Then $[k,g]$ moves almost $R$-maximally and $[g,k]$ moves almost $L$-maximally.
\end{proof}

We can now prove part (i) of Theorem \ref{main}.
\begin{proof}[Proof of part (i) of Theorem \ref{main}]
Let $\Mlin$ be the linear order expansion of a non-trivial free homogeneous structure. Let $g$ be a non-trivial automorphism of $\Mlin$. Then, by Theorem \ref{oneorbitalgeneral}, there exists $h\in Aut(\Mlin)$ as a product of conjugates of $g$ and $g^{-1}$, such that $ha <a$ for all $a \in \Mlin$. By the previous lemma, there exists $k \in Aut(\Mlin)$ such that $[k,h]$ moves almost $R$-maximally and $[h,k]$ moves almost $L$-maximally. Therefore, by Theorem \ref{swirthm}, any element of $Aut(\Mlin)$ can be written as a product of conjugates of $g$ and $g^{-1}$. Thus, $Aut(\Mlin)$ is simple.
\end{proof}

\section{Universal $n$-linear order}

\begin{dfn}
Let $\lan=\{<_1,...,<_n\}$. Let $\cl$ be the class of all finite $\lan$-structures, where $<_i$ is interpreted as a linear order for each $i =1,...,n$. We call the Fra{\"\i}ss{\'e} limit of $\cl$, denoted by $\M_n$, the \emph{universal $n$-linear orders}.
\end{dfn}

\begin{rmk}
For $n=1$, it is the dense linear order $(\rat, <)$. For $n=2$, it is called a \emph{universal permutation} in \cite{cameron2002homohpermutation} as a permutation can be interpreted as two linear orders on the underlying set.
\end{rmk}

By repeated application of Theorem \ref{swirfromswir}, starting with the trivial countable homogeneous structure, we can find a $\mathbf{SWIR}$ on the universal $n$-linear order. 

\begin{coro}
Let $\M_n$ be the universal $n$-linear order. For finite substructures $A,B,C\subseteq \M_n$, define $A \ind_B C$ if for any $a \in A \setminus B, c\in C \setminus B$ such that $a < _i c$ for some $i$, there exists $b\in B$ such that $a < _i b< _i c$. Then $\ind$ is a $\mathbf{SWIR}$ on $\M_n$. 
\end{coro}

We first look at the case where $n=1$. The normal subgroup structure of $Aut(\rat, \leq)$ is well-known. According to \cite{ball1985normal}, Higman showed that the subgroup consisting of all bounded (both left-bounded and right-bounded) automorphisms is simple in \cite{higman1951simple}. \cite{holland1963lattice} and \cite{lloyd1964lattice} studied the normal subgroup generated by unbounded automorphisms. We illustrate in the following how Theorem \ref{swirthm} can be applied to obtain an alternative proof of the same statement.

\begin{coro}
Let $g$ be an unbounded automorphism of $(\rat,\leq)$. Then the normal subgroup generated by $g$ is $Aut(\rat,\leq)$.
\end{coro}
\begin{proof}
Let $g$ be an unbounded automorphism of $(\rat,\leq)$. By Theorem \ref{oneorbitalgeneral}, there exists $h\in Aut(\rat,\leq)$, a product of conjugates of $g$ and $g^{-1}$, such that $ha<a$ for all $a \in \rat$. Let $p$ be any type over some finite set $X\subseteq \rat$. Then by Lemma \ref{move}, we can find $\bar{a}$ realising $p$ such that if $a_i < ga_j$ for some $a_i \in \bar{a}\setminus X, ga_j \in g\bar{a}\setminus X$, then there exists $x \in X$ such that $a_i <x<ga_j$. Hence $\bar{a} \ind_X h\bar{a}$ by definition. Similarly, we can find $\bar{a}'$ realising $p$ such that $h^{-1}\bar{a}' \ind_X \bar{a}'$. So, $h$ moves $R$-maximally and $h^{-1}$ moves $L$-maximally. Therefore, by Theorem \ref{swirthm}, the normal subgroup generated by $g$ is $Aut(\rat,\leq)$ .
\end{proof}

For $n\geq 2$, we will use a similar approach to find an automorphism that moves almost $R$-maximally and its inverse moves almost $L$-maximally. We will do this by induction. We first define the following notion and prove the following lemmas.

\begin{dfn}
Let $\M_n$ be the universal $n$-linear order and $I \subseteq \{1,\ldots, n\}$. Define $\M_n^I$ to be the structure $\M_n$ but with the ordering $<_i$ reversed for every $i \in I$.
\end{dfn}
\begin{rmk}
Then we have $\M_n^{I} \cong \M_n$ for any $I \subseteq \{1,...,n\}$. Hence, $Aut(\M_n) = Aut(\M_n^I)$.
\end{rmk}

We will also need the following lemmas to prove the next theorem.
\begin{lem}\label{EPfornorder}
Let $\M_n$ be the universal $n$-linear order and $a_1,...,a_n,c_1,...,c_n \in \M_n$ be such that $a_i<_ic_i$ for all $i=1,...,n$. Then there exists $b \in \M_n$ such that $a_i<_ib<_ic_i$ for all $i=1,...,n$.
\end{lem}
\begin{proof}
This follows from the Extension Property: we can embed $a_1...a_nc_1...c_n\\ \in \M_n$ into $a_1...a_nc_1...c_nb' \in \cl$ such that $a_i<_ib<_ic_i$ for all $i=1,...,n$, where $\cl$ is the class of all finite $n$-linearly ordered structures. Then by the Extension Property, we can embed $a_1...a_nc_1...c_nb'$ into $\M_n$. Hence, there exists $b \in \M_n$ such that $a_i<_ib<_ic_i$ for all $i=1,...,n$.
\end{proof}

\begin{lem}
Suppose $g,g' \in Aut(\M_n)$ have a single $+$-orbital on $<_i$ for some $i=1,...,n$. Then 
\begin{enumerate}[(i)]
\item $g'g$ has a single $+$-orbital on $<_i$.
\item Any product of conjugates of $g$ has a single $+$-orbital on $<_i$.
\end{enumerate}
Similarly, any product of conjugates of automorphisms that has a single $-$-orbital on $<_i$ has a single $-$-orbital on $<_i$.
\end{lem}
\begin{proof}
Since $g$ has a single $+$-orbital on $<_i$, there exists $(a_j)_{j \in \mathbb{Z}}$ such that $ga_j=a_{j+1}>_ia_j$ and $\cup_{j \in \mathbb{Z}} (a_j,a_{j+1})_i=\M_n$. 
\begin{enumerate}[(i)]
\item Then for any $j \in \mathbb{Z}$, since $g'x>_ix$ for all $x \in \M_n$, we have $g'g a_j =g'a_{j+1}>_ia_{j+1}$. Hence $(a_j,g'ga_{j})_i$ overlaps with $(a_{j+1},g'ga_{j+1})_i$ for all $j \in \mathbb{Z}$. Since $\cup_{j \in \mathbb{Z}} (a_j,a_{j+1})_i=\M_n$, $g'g$ has a single $+$-orbital on $<_i$. 
\item Let $h \in Aut(M_n)$. Then for all $j \in \mathbb{Z}$, $hgh^{-1}(ha_j)=hga_j=ha_{j+1}$. Since $a_{j+1}>_ia_j$, we have $ha_{j+1}>_iha_j$. Since $\cup_{j \in \mathbb{Z}} (a_j,a_{j+1})_i=\M_n$, we also have  $\cup_{j \in \mathbb{Z}} (ha_j,ha_{j+1})_i=\M_n$. Hence $hgh^{-1}$ has a single $+$-orbital on $<_i$. By part (i), any product of conjugates of $g$ has a single $+$-orbital on $<_i$.
\end{enumerate}
\end{proof}

To simplify some notions in the following proof, we denote $(a,b)_i$ for some $a,b\in \M_n$ and $i \in \{1,...,n\}$ to be the open interval $(a,b)$ on $<_i$ if $a<_ib$ and $(b,a)$ on $<_i$ if $b<_ia$.

\begin{thm}\label{universallinearorder}
Let $g \in Aut (\M_n)$, $n\geq 2$, be non-trivial. Then there exists $f$, a product of conjugates of $g$ and $g^{-1}$, and $I\subseteq \{1,...,n\}$ such that $f$ has a single $+$-orbital on $<_j$ for all $j=1,...,n$ in $\M_n^I$. 
\end{thm}
\begin{proof}
We prove the theorem by induction on $n$. For the inductive base, by Theorem \ref{oneorbitalgeneral}, there exists $h' \in Aut(\M_n)$, a product of conjugates of $g$ and $g^{-1}$, such that $h'$ has a single $+$-orbital on $<_1$ in $\M_n^{I_1}$, where $I_1=\emptyset$.

Now for the inductive step, by the inductive hypothesis, suppose $h\in Aut(\M_n^{I_{m-1}})$ is such that $h$ can be written as a product of conjugates of $g$ and $g^{-1}$ and has a single $+$-orbital on $<_j$ for all $j=1,...,m-1$ in $\M_n^{I_{m-1}}$ for some $I_{m-1} \subseteq \{1,...,m-1\}$. Since $\M_n^{I} \cong \M_n$, we may assume without loss of generality, that $I_{m-1}=\emptyset$. We will show that there exists $k\in Aut(\M_n)$, a product of conjugates of $h$ and $h^{-1}$, such that $k$ has a single $+$-orbital on $<_j$ for all $i=1,...,m$ in $\M_n^{I_{m}}$ for some $I_{m} \subseteq \{1,...,m-1\}$. There are four possible cases to consider. 

Case I. For any $x\in \M_n$, there exists $y<_mx$ and $z>_mx$ such that $hy>_my$ and $hz>_mz$.

In this case, we can apply Proposition \ref{oneorbitalcase1} and find $k_1 \in Aut(\M_n)$ such that $hk_1hk_1^{-1}$ has a single $+$-orbital on $<_m$. By the previous lemma, $hk_1hk_1^{-1}$ has a single $+$-orbital on $<_j$ for $j=1,...,m-1$. Then $k=hk_1hk_1^{-1}$ has a single $+$-orbital on $<_j$ for all $i=1,...,m$ in $\M_n$, where $I_{m}=I_{m-1}=\emptyset$.

Case II. For any $x\in \M_n$, there exists $y<_mx$ and $z>_mx$ such that $hy<_my$ and $hz<_mz$.

Let $I_m=\{m\}$. Then for any $x\in \M_n^{I_{m}}$, there exists $y<_mx$ and $z>_mx$ such that $hy>_my$ and $hz>_mz$. We can again apply Proposition \ref{oneorbitalcase1} to $h$ and find $k_1 \in Aut(\M_n^{I_{m}})$ such that $hk_1hk_1^{-1}$ has a single $+$-orbital on $<_m$. By the same argument as in case I, $k=hk_1hk_1^{-1}$ has a single $+$-orbital on $<_j$ for all $i=1,...,m$ in $\M_n^{I_{m}}$.

Case III. For any $x\in \M_n$, there exists $y>_mx$ and $z<_mx$ such that $hy<_my$ and $hz>_mz$. 

Let $I_m=I_{m-1}=\emptyset$. By Proposition \ref{2orbitals}, we can find $k_1 \in Aut(\M_n)$ such that $h^{-1}k_1^{-1}h^{-1}k_1$ has a $+$-orbital unbounded above and a $-$-orbital unbounded below with respect to $<_m$. Let $k_2=k_1^{-1}hk_1h$. Then there exists $y,z \in \M_n$ such that $k_2a<_ma$ for all $a>_my$ and $k_2a>_ma$ for all $a<_mz$. By the previous lemma, we also have that $k_2$ has a single $+$-orbital on $<_j$ for all $j=1,...,m-1$. We want to construct $k_3$ using the back-and-forth method such that $[k_2,k_3]$ has a single $+$-orbital on $<_j$ for all $j=1,...,m$. We start with $\tilde{k}$ as the empty map and extend $\tilde{k}$ so it approximates to $k_3$.

\begin{figure}[h]
\centering
\begin{tikzpicture}
  \draw (0,0) node (1a){}
++(0.5,0) node[draw,circle,fill=black,minimum size=4pt,inner sep=0pt] (1) {}
++(0,0) node[above] (1n)  {$a_{-s-1}$}
     ++(1.0,0) node[draw,circle,fill=black,minimum size=4pt,inner sep=0pt] (2){} 
++(0,0) node[below] (2n)  {$a_{-s}$}
  ++(1.0,0) node[draw,circle,fill=black,minimum size=4pt,inner sep=0pt] (3) {}
     ++(0,0) node[above] (3x)  {$k_2a_{-s}$}
   ++(1.5,0) node[draw,circle,fill=black,minimum size=4pt,inner sep=0pt] (4) {}
   ++(0,0) node[above] (4n)  {$a_{-1}$}
     ++(0.9,0) node[draw,circle,fill=black,minimum size=4pt,inner sep=0pt] (5) {}
     ++(0,0) node[above] (5n)  {$a_0$}
  ++(1.3,0) node[draw,circle,fill=black,minimum size=4pt,inner sep=0pt] (7) {}
     ++(0,0) node[above] (7n)  {$a_1$}
     ++(0,0) node[below]  {$=k_2^{-1}c_0$}
  ++(1.5,0) node[draw,circle,fill=black,minimum size=4pt,inner sep=0pt] (6) {}
     ++(0,0) node[above] (6n)  {$k_2a_{0}$}
 ++(0.9,0) node[draw,circle,fill=black,minimum size=4pt,inner sep=0pt] (8)  {}
    ++(0,0) node[above] (8n)  {$c_{0}$}
 ++(1.2,0) node[draw,circle,fill=black,minimum size=4pt,inner sep=0pt] (10)  {}
    ++(0,0) node[above] (10n)  {$a_{s}$}
++(1.2,0) node[draw,circle,fill=black,minimum size=4pt,inner sep=0pt] (11)  {}
    ++(0,0) node[above] (11n)  {$c_{s}$}
        ++(0.8,0) node (11a)  {}
++(0,-2) node (12a)  {}
++(-1.1,0) node[draw,circle,fill=black,minimum size=4pt,inner sep=0pt] (12){} 
++(0,0) node[below] (12n)  {$k_2b_s$}
++(-1.2,0) node[draw,circle,fill=black,minimum size=4pt,inner sep=0pt] (13){} 
++(0,0) node[below] (13n)  {$b_s$}
++(-1.8,0) node[draw,circle,fill=black,minimum size=4pt,inner sep=0pt] (15){} 
++(0,0) node[below] (15n)  {$k_2b_0$}
++(-1.5,0) node[draw,circle,fill=black,minimum size=4pt,inner sep=0pt] (16){} 
++(0,0) node[below] (16n)  {$b_{-1}$}
++(-1.3,0) node[draw,circle,fill=black,minimum size=4pt,inner sep=0pt] (17){} 
++(0,0) node[below] (17n)  {$b_{0}$}
++(-1.2,0) node[draw,circle,fill=black,minimum size=4pt,inner sep=0pt] (18){} 
++(0,0) node[below] (18n)  {$k_2^{-1}b_{-1}$}
++(-1.5,0) node[draw,circle,fill=black,minimum size=4pt,inner sep=0pt] (19){} 
++(0,0) node[below] (19n)  {$b_{-s-1}$}
++(-1.2,0) node[draw,circle,fill=black,minimum size=4pt,inner sep=0pt] (20){} 
++(0,0) node[below] (20n)  {$k^{-1}_2b_{-s-1}$}
++(-1.0,0) node (19a){};

\draw[->] (1) -- (20) node[draw=none,fill=none,font=\scriptsize,midway,right]{};
\draw[->] (3) -- (19) node[draw=none,fill=none,font=\scriptsize,midway,right]{};
\draw[->] (5) -- (17) node[draw=none,fill=none,font=\scriptsize,midway,right]{};
\draw[->] (4) -- (18) node[draw=none,fill=none,font=\scriptsize,midway,right]{};
\draw[->] (8) -- (15) node[draw=none,fill=none,font=\scriptsize,midway,right]{};
\draw[->] (6) -- (16) node[draw=none,fill=none,font=\scriptsize,midway,right]{};
\draw[->] (10) -- (13) node[draw=none,fill=none,font=\scriptsize,midway,right]{};
\draw[->] (11) -- (12) node[draw=none,fill=none,font=\scriptsize,midway,right]{};
\draw[-] (1a) -- (11a) node[draw=none,fill=none,font=\scriptsize,midway,right]{};
\draw[-] (19a) -- (12a) node[draw=none,fill=none,font=\scriptsize,midway,right]{};
\end{tikzpicture}
\caption{Construction of $k_3$ on $<_j$, where $j=1,...,m-1$} \label{fig:c3o1}
\begin{tikzpicture}
  \draw (0,0) node (1a) {}
     ++(0.5,0) node[draw,circle,fill=black,minimum size=4pt,inner sep=0pt] (1){} 
   ++(0,0) node[above] (1n)  {$a_{-s-1}$}
     ++(0.8,0) node[draw,circle,fill=black,minimum size=4pt,inner sep=0pt] (2){} 
   ++(0,0) node[above] (2n)  {$a_{-s}$}
   ++(1.0,0) node[draw,circle,fill=black,minimum size=4pt,inner sep=0pt] (3)  {}
++(0,0) node[above] (3n)  {$k_2a_{-s}$}
   ++(1.5,0) node[draw,circle,fill=black,minimum size=4pt,inner sep=0pt] (4) {}
   ++(0,0) node[above] (4n)  {$a_{-1}$}
  ++(0.8,0.5) node (6n)  {}
  ++(0,-3) node (6m) {}
  ++(1,0) node (7m) {}
 ++(0,3) node (7n)  {}
  ++(-0.2,-.5) node[above] (7a)  {$y$}
    ++(-1,0) node[above] (6a)  {$z$}
 ++(2.2,0) node[draw,circle,fill=black,minimum size=4pt,inner sep=0pt] (5)  {}
    ++(0,0) node[above] (8n)  {$k_2a_0$}
     ++(0.8,0) node[draw,circle,fill=black,minimum size=4pt,inner sep=0pt] (8) {}
     ++(0,0) node[above] (5n)  {$c_{0}$}
 ++(0.8,0) node[draw,circle,fill=black,minimum size=4pt,inner sep=0pt] (9)  {}
    ++(0,0) node[above] (9n)  {$a_{0}$}
     ++(1.0,0) node[draw,circle,fill=black,minimum size=4pt,inner sep=0pt] (20) {}
     ++(0,0) node[below]  {$=k_2^{-1}c_0$}
     ++(0,0) node[above]  {$a_{1}$}
 ++(1.0,0) node[draw,circle,fill=black,minimum size=4pt,inner sep=0pt] (10)  {}
    ++(0,0) node[above] (10n)  {$c_{s}$}
 ++(0.6,0) node[draw,circle,fill=black,minimum size=4pt,inner sep=0pt] (21)  {}
    ++(0,0) node[above] (21n)  {$a_s$}
 ++(0.9,0) node[draw,circle,fill=black,minimum size=4pt,inner sep=0pt]   {}
     ++(0,0) node[below]  {$=k_2^{-1}c_s$}
     ++(0,0) node[above]  {$a_{s+1}$}
        ++(0.5,0) node (11a)  {}
++(0,-2) node (12a)  {}
++(-0.8,0) node[draw,circle,fill=black,minimum size=4pt,inner sep=0pt] (12){} 
++(0,0) node[below] (12n)  {$b_s$}
++(-1.2,0) node[draw,circle,fill=black,minimum size=4pt,inner sep=0pt] (13){} 
++(0,0) node[below] (13n)  {$k_2b_s$}
++(-1.9,0) node[draw,circle,fill=black,minimum size=4pt,inner sep=0pt] (14){} 
++(0,0) node[below] (14n)  {$b_{0}$}
++(-1.5,0) node[draw,circle,fill=black,minimum size=4pt,inner sep=0pt] (15){} 
++(0,0) node[below] (15n)  {$k_2b_{0}$}
++(-2.7,0) node[draw,circle,fill=black,minimum size=4pt,inner sep=0pt] (18){} 
++(0,0) node[below] (18n)  {$b_{-1}$}
++(-1.0,0) node[draw,circle,fill=black,minimum size=4pt,inner sep=0pt] (19){} 
++(0,0) node[below] (19n)  {$k_2^{-1}b_{-1}$}
++(-1.0,0) node[draw,circle,fill=black,minimum size=4pt,inner sep=0pt] (22){} 
++(0,0) node[below] (22n)  {$b_{-s-1}$}
++(-1.2,0) node[draw,circle,fill=black,minimum size=4pt,inner sep=0pt] (23){} 
++(0,0) node[below] (22n)  {$k^{-1}_2b_{-s-1}$}
++(-0.9,0) node (19a){};

\draw[->] (1) -- (23) node[draw=none,fill=none,font=\scriptsize,midway,right]{};
\draw[->] (3) -- (22) node[draw=none,fill=none,font=\scriptsize,midway,right]{};
\draw[->] (8) -- (15) node[draw=none,fill=none,font=\scriptsize,midway,right]{};
\draw[->] (4) -- (19) node[draw=none,fill=none,font=\scriptsize,midway,right]{};
\draw[->] (5) -- (18) node[draw=none,fill=none,font=\scriptsize,midway,right]{};
\draw[->] (10) -- (13) node[draw=none,fill=none,font=\scriptsize,midway,right]{};
\draw[->] (9) -- (14) node[draw=none,fill=none,font=\scriptsize,midway,right]{};
\draw[->] (21) -- (12) node[draw=none,fill=none,font=\scriptsize,midway,right]{};
\draw[-] (1a) -- (11a) node[draw=none,fill=none,font=\scriptsize,midway,right]{};
\draw[-] (19a) -- (12a) node[draw=none,fill=none,font=\scriptsize,midway,right]{};
\draw[-] (6n) -- (6m) node[draw=none,fill=none]{};
\draw[-] (7n) -- (7m) node[draw=none,fill=none]{};
\end{tikzpicture}
\caption{Construction of $k_3$ on $<_m$} \label{fig:c3o2}
\end{figure}

List all elements of $\M_n$ as $w_0,w_1,...$. Choose $a_0>_m y,x_0$. By Lemma \ref{EPfornorder}, we can find $b_0 \in \M_n$ such that $a_0<_jb_0<_j k_2a_0$ for $j=1,...,m-1$ and $b_0 >_m y,x_0$. Then $k_2b_0<_mb_0$ and $(a_0,k_2a_0)_j \cap (b_0,k_2b_0)_j\neq \emptyset$ for all $j=1,...,m-1$. Extend $\tilde{k}$ by sending $a_0$ to $b_0$. Since $\tilde{k}^{-1}\cdot tp(k_2b_0/b_0) \cup \{x >_{j} k_2a_0|j=1,...,m\}$ is consistent, we can find $c_0 >_{j} k_2a_0$ for all $j=1,...,m$ realising $\tilde{k}^{-1}\cdot tp(k_2b_0/b_0)$. Extend $\tilde{k}$ by sending $c_0$ to $k_2b_0$. Let $a_1:=k_2^{-1}c_0$. Then for all $j=1,...,m$, we have $a_1=k_2^{-1}c_0=[k_2,\tilde{k}] a_0>_{j}a_0$ since $c_0 >_{j} k_2a_0$. Similarly, since $\tilde{k}\cdot tp(k_2a_0/c_0a_0) \cup \{ x<_mz,x_0 \}$ is consistent, we can choose $b_{-1}<_mz,x_0$ realising $\tilde{k}\cdot tp(k_2a_0/c_0a_0)$. Choose $a_{-1}<_mz,x_0$ realising $\tilde{k}^{-1}\cdot  tp(k_2^{-1}b_{-1}/b_{-1}b_0k_2b_0)$. Then for all $j=1,...,m$, we have $ a_{-1}=[k_2,\tilde{k}] ^{-1} a_0<_{j} [k_2,\tilde{k}]^{-1} a_1=a_0$ since $a_0<_{j}a_1$.

Let $A_0=\{a_{-1},k_2a_0,c_0,a_0\}$ and $B_0=\{k_2^{-1}b_{-1},b_{-1},k_2b_0,b_0\}$. Since $S:=\bigcap _{j=1}^m (\min_j A_0 ,\max_jA_0)_j \cap \min_j B_0 ,\max_jB_0)_j$ is non-empty, we can let $w_0$ be the first element in the list $w_0,w_1,...$ such that $x_0$ is in $S$. Choose $y_0$ realising $\tilde{k}\cdot tp(x_0/A_0)$ and extend $\tilde{k}$ by sending $x_0$ to $y_0$. Choose $z_0$ realising $\tilde{k}^{-1}\cdot tp(x_0/y_0B_0)$ and extend $\tilde{k}$ by sending $z_0$ to $x_0$. Extend $A_0$,$B_0$ to include $x_0,z_0$ and $y_0,x_0$ respectively.

Now suppose at the $s$-th step, we have a partial isomorphism of $\M_n$, $\tilde{k}: A_{s-1} \rightarrow B_{s-1}$ satisfying
\begin{enumerate}[(i)]
\item $A_{s-1}= \{a_{-s},k_2a_{-s+1},...,c_{s-1},a_{s-1},x_0,...,x_{s-1},z_0,...,z_{s-1}\}$, \\ $B_{s-1} = \{k_2^{-1}b_{-s},b_{-s},....,k_2b_{s-1},b_{s-1},x_0,...,x_{s-1},y_0,...,y_{s-1}\}$ 
\item $\tilde{k}$ maps $a_i$ to $b_i$, $c_i$ to $k_2b_i$, $x_i$ to $y_i$ and $z_i$ to $x_i$ for all $i=0,...,s-1$ and $\tilde{k}$ maps $a_{-i}$ to $k_2^{-1}b_{-i}$ and $k_2a_{-i+1}$ to $b_{-i}$ for all $i=1,...,s$,
\item $a_i=k_2^{-1}c_{i-1}$ for $i=1,...,s$ and $[k_2,\tilde{k}] a_{i}=a_{i+1}>_ja_{i}$ for all $i=-s,...,s-1$ and $j=1,...,m$,
\item on order $<_j$, where $j=1,...,m-1$, we have 
\begin{enumerate}[(a)]
\item $a_{i}>_j k_2^{i-1}a_0$, $a_{-i}<_j k_2^{-i+1}a_0$, $b_i>_j k_2^{\lfloor\frac{i}{2}\rfloor-1}b_0$, $b_{-i}<_j k_2^{-\lfloor\frac{i}{2}\rfloor+1}b_0$ for $i=0,...,s-1$,
\item $\min_j A_{s-1}=a_{-s}<_j k_2a_{-s}<_j a_{-s+1} <_j k_2 a_{-s+1}<_j \cdots <_j a_{s-1}<_j c_{s-2}<_j a_s <_j c_{s-1} =\max_j A_{s-1}$, 
\item $\min_j B_{s-1}=k_2^{-1}b_{-s}<_j k_2^{-1} b_{-s+1}<_j b_{-s} \cdots <_j b_{s-1}<_j k_2b_{s-2}<_j k_2b_{s-1}=\max_j B_{s-1}=k_2b_{s-1}$.
\end{enumerate}
\item on $<_m$, we have 
\begin{enumerate}[(a)]
\item $a_i>_m k_2^{-\lfloor\frac{i}{2}\rfloor+1}a_0$, $a_{-i-1}<_m k_2^{-i+1}a_{-1}$, $b_{i}>_mk_2^{-i+1}b_0$, $b_{-i-1}<_mk_2^{-\lfloor\frac{i}{2}\rfloor+1}b_{-1}$ for $i=0,...,s-1$, 
\item $\min_j A_{s-1}=a_{-s}<_m k_2a_{-s}<_m a_{-s+1} <_m k_2 a_{-s+1}<_m \cdots <_m a_{s-2} <_m c_{s-1}<_m a_{s-1}=\max_m A_{s-1}$, and 
\item $\min_j B_{s-1}=k_2^{-1}b_{-s}<_m k_2^{-1} b_{-s+1}<_m b_{-s} <_m\cdots <_m k_2b_{s-1}<_m b_{s-1}=\max_j B_{s-1}$,
\end{enumerate}
\end{enumerate}

We can choose $b_s$ such that $b_s >_m k_2^{-s+1}b_0, k_2^{-1}b_{s-1}$ and $b_s$ realises $\tilde{k}\cdot tp(a_s/A_{s-1})$ since the type $\tilde{k}\cdot tp(a_s/A_{s-1}) \cup \{x>_m k_2^{-s+1}b_0, k_2^{-1}b_{s-1}\}$ is consistent. 
Extend $\tilde{k}$ by sending $a_s$ to $b_s$. 
Similarly, we can find $c_{s}$ such that $c_s>_j k_2c_{s-1}, k_2^{s+1} a_0$ for all $j=1,...,m-1$ and $c_s$ realises $\tilde{k}^{-1}\cdot tp(k_2b_s/b_sB_{s-1})$. Extend $\tilde{k}$ by sending $c_{s}$ to $k_2b_s$. Let $a_{s+1}:=k_2^{-1}c_{s}$. Then we have the following:
\begin{enumerate}[(i)]
\item for all $j=1,...,m$, we have $a_{s+1}=[k_2,\tilde{k}] a_s>_j [k_2,\tilde{k}] a_{s-1}=a_s$ since $a_s>_ja_{s-1}$.
\item Since for any $j=1,...,m-1$, $c_{s}>_j  k_2^{s+1} a_0,k_2c_{s-1}$, we have $a_{s+1}=k^{-1}_2 c_s>_j k_2^{s} a_0, c_{s-1} $. 
\item For $j=1,...,m-1$, since by the inductive hypothesis that $a_s>_j c_{s-2}$ and $\tilde{k}$ maps $a_s$ to $b_s$ and $c_{s-2}$ to $k_2b_{s-2}$, we have $b_s>_j k_2b_{s-2}$. By the inductive hypothesis that $b_{s-2}>_m k_2^{\lfloor\frac{s-2}{2}\rfloor-1}b_0$, we have $b_s>_j k_2^{\lfloor\frac{s}{2}\rfloor-1}b_0$.
\item Since $b_s>_m k_2^{-1}b_{s-1}$, we have $k_2 b_s>_m b_{s-1}$. Since $\tilde{k}$ maps $c_s$ to $k_2b_s$ and $a_{s-1}$ to $b_{s-1}$, we have $c_s>_m a_{s-1}$. By the inductive hypothesis that $a_{s-1}>_m k_2^{-\lfloor\frac{s-1}{2}\rfloor+1}a_0$, we have $a_{s+1}=k^{-1}_2 c_s >_mk^{-1}_2a_{s-1}>_mk_2^{-\lfloor\frac{s-1}{2}\rfloor}a_0$.
\end{enumerate}

Similarly, we can choose $b_{-s-1}$ such that $b_{-s-1}<_j k_2 b_{-s}, k_2^{s}b_0$ for all $j=1,...,m-1$ and $b_{-s-1}$ realises $\tilde{k}\cdot tp(k_2a_{-s}/a_s c_{s}A_{s-1})$. Extend $\tilde{k}$ by sending $k_2a_{-s}$ to $b_{-s-1}$. Choose $a_{-s-1}<_m k^{-1}_2 a_{-s}, k_2^{-s}a_0$ realising $\tilde{k} ^{-1}\cdot tp(k_2^{-1}b_{-s-1}/b_sb_{-s-1}k_2b_sB_{s-1})$. 

Choose $b_{-s-1}<_m k_2 b_{-s}$ realising $\tilde{k}\cdot tp(k_2a_{-s}/a_sc_sA_{s-1})$. Extend $\tilde{k}$ by sending $k_2a_{-s}$ to $b_{-s-1}$. Choose $a_{-s-1}$ such that $a_{-s-1}<_j k^{-1}_2 a_{-s}, k_2^{-s}a_0$ for all $j=1,...,m$ and $a_{-s-1}$ realises $\tilde{k}^{-1}\cdot  tp(k_2^{-1}b_{-s-1}/b_sb_{-s-1}k_2b_sB_{s-1})$. Extend $\tilde{k}$ by sending $a_{-s-1}$ to $k_2^{-1}b_{-s-1}$. Then we have the following for all $j=1,...,m$:
\begin{enumerate}[(i)]
\item $a_{-s-1}=[k_2,\tilde{k}] ^{-1} a_{-s}<_j [k_2,\tilde{k}] ^{-1}a_{-s+1}=a_{-s}$ since $a_{-s}<_j a_{-s+1}$.
\item By the inductive hypothesis that $k_2a_{-s}<_j a_{-s+1}$ and since $\tilde{k}$ maps $k_2a_{-s}$ to $b_{-s-1}$ and $a_{-s+1}$ to $k^{-1}_2b_{-s+1}$, we have $b_{-s-1}<_j k^{-1}_2b_{-s+1}$. By the inductive hypothesis that $b_{-s+1}<_j k_2^{-\lfloor\frac{s-1}{2}\rfloor+1}b_0$, we have $b_{-s-1}<_j k_2^{-\lfloor\frac{s+1}{2}\rfloor+1}b_0$. 
\item $k_2a_{-s-1}<_j a_{-s}$ since $a_{-s-1}<_j k^{-1}_2 a_{-s}$.
\end{enumerate}

Let $A_s=A_{s-1}\cup \{a_s,c_s,k_2a_{-s},a_{-s-1}\}$ and $B_s=B_{s-1}\cup\{k_2^{-1}b_{-s-1}, \\b_{-s-1},k_2b_s,b_s\}$. Let $x_s$ be the first element in the list $w_0,w_1,....$ such that $x_s \notin \{x_0,...,x_{s-1}\}$ and $x_s$ is in $\bigcap _{i=1}^m ( \min_iA_s, \max_iA_s)_i \cap (\min_iB_s,\max_iB_s)_i$. Find $y_s$ realising $\tilde{k}\cdot tp(x_s/A_s)$ and extend $\tilde{k}$ by sending $x_s$ to $y_s$. Find $z_s$ realising $\tilde{k}^{-1}\cdot tp(x_s/y_sB_s)$. Extend $A_s$,$B_s$ to include $x_s,z_s$ and $y_s,x_s$ respectively. Then, $\tilde{k}:A_s \rightarrow B_s$ satisfies the inductive hypothesis (i)-(v).

Let $k_3$ be the union of $\tilde{k}$ over all steps. 
Since $k_2$ has a singe $+$-orbital on $<_j$ for all $j=1,...,m-1$ and a $-$-orbital unbounded above and a $+$-orbital unbounded below on $<_m$. 
By Remark \ref{orbitalrmk}, $(k_2^ia_0)_{i\in \mathbb{Z}}, (k_2^ib_0)_{i\in \mathbb{Z}}$ are unbounded above and below on $<_j$ for all $j=1,...,m-1$. $(k_2^ia_0)_{i\in \mathbb{Z}}, (k_2^ib_0)_{i\in \mathbb{Z}}$ are unbounded above on $<_m$ and $(k_2^ia_{-1})_{i\in \mathbb{Z}}, (k_2^ib_{-1})_{i\in \mathbb{Z}}$ are unbounded below on $<_m$ 
Then by the inductive hypothesis (iv.a) and (v.a), we know that 
\[ \bigcup_{n \in \mathbb{N}} \big( \cap_{j=1}^m ( \min_{<_j} A_n, \max_{<_j} A_n)_j \cap(\min_{<_j} B_n,\max_{<_j} B_n)_j \big)=\Mlin.\] 
Thus, for any $w_i$ in the list $w_0, w_1,...$, $w_i=x_j$ for some $j \in \mathbb{N}$, i.e. $w_i$ is in the domain and image of $k_3$. Hence, $k_3$ is bijective and $k_3 \in Aut(\M_n^{I_{m}})$. Since $([k_2,k_3]^ia_0)_{i\in \mathbb{Z}}$ is unbounded above and below on $<_j$ for all $j=1,...,m$, $[k_2,k_3]$ has a single $+$-orbital on $<_j$ for all $j=1,...,m$ on $\M_n^{I_{m}}$.

Case IV.  For any $x\in \M_n^{I_{m-1}}$, there exists $y>_mx$ and $z<_mx$ such that $hy>_my$ and $hz<_mz$. 

Let $I_m=I_{m-1}\cup \{m\}$. Then for any $x\in \M_n^{I_{m}}$, there exists $y>_mx$ and $z<_mx$ such that $hy<_my$ and $hz>_mz$. By Proposition \ref{2orbitals}, we can find $k_1 \in Aut(\M_n^{I_{m}})$ such that $h^{-1}k_1^{-1}h^{-1}k_1$ has a $+$-orbital unbounded above and a $-$-orbital unbounded below with respect to $<_m$. Let $k_2=k_1^{-1}hk_1h$. Then by the same argument as in case III, there exists $k_3 \in Aut(\M_n^{I_{m}})$ such that $[k_2,k_3]$ has a single $+$-orbital on $<_j$ for all $j=1,...,m$ on $\M_n^{I_{m}}$.

Therefore, by induction, we can find $f \in Aut(\M_n^I)$, as a product of conjugate of $g$ and $g^{-1}$, such that $f$ has a single $+$-orbital on $<_j$ for all $j=1,...,n$ in $\M_n^I$ for some $I\subseteq \{1,...,n\}$. 
\end{proof}

We now show that for an automorphism $g \in \M_n$, if $ga <_{i} a$ for all $a\in \M_n$ and $i=1,...,n$, then $g$ moves almost $R$-maximally $g^{-1}$ moves almost $L$-maximally. In order to make the notions simpler in the proof of the next theorem, we can define $\bar{a}\ind^{<_k}_X \bar{b}$ if for any $a_i \in \bar{a} \setminus X$, $b_j\in\bar{b}\setminus X$ where $a_i <_{k} ga_j$, there exists $x \in X$ such that $a_i<_kx<_kb_j$. Then by definition, we can prove the following lemma.

\begin{lem}
Given $\bar{a},\bar{b} \in \M_n$, define $\bar{a}\ind^{<_k}_X \bar{b}$ if for some $a_i \in \bar{a} \setminus X$, $b_j\in\bar{b}\setminus X$, $a_i <_{k} ga_j$. Then,
\begin{enumerate}
\item $\bar{a} \ind_X \bar{b}$ if and only if $\bar{a}\ind^{<_k}_X \bar{b}$ for all $k=1,...,n$.
\item if $\bar{a} \ind^{<_k}_X \bar{b}$ and $\bar{a}' \ind^{<_k}_X \bar{b}$, then $\bar{a}\bar{a}' \ind^{<_k}_X \bar{b}$.
\item if $\bar{a} \ind^{<_k}_X \bar{b}$ and $\bar{a} \ind^{<_k}_X \bar{b}'$, then $\bar{a} \ind^{<_k}_X \bar{b}\bar{b}'$.
\end{enumerate}
\end{lem}

\begin{thm}
Let $g \in Aut(\M_n)$. If $ga <_{i} a$ for any $ a \in \M_n$ and $i=1,...,n$, $g$ moves almost $R$-maximally. Similarly if $ga >_{i} a$ for any $a \in \M_n$ and $i=1,...,n$, then $g$ moves almost $L$-maximally. 
\end{thm}
\begin{proof}
Let $p=tp(\bar{x}/X)$ be an $l$-type over $X$. Without loss of generality, we may assume $x_i \notin X$ for any $x_i\in \bar{x}$. Firstly, by Lemma \ref{move}, we can find $\bar{a}$ realising $p$ such that $\bar{a} \ind^{<_1}_X g\bar{a}$. Suppose we can find $\bar{b}$ realising $p$ such that $\bar{b} \ind^{<_i}_X g\bar{b}$ for all $i=1,...,m-1$. We will show that we can find $\bar{c}$ realising $p$ such that $\bar{c} \ind^{<_i}_X g\bar{c}$ for $i=1,...,m$. Therefore, by induction on $m$, we can find $\bar{d}$ realising $p$ such that $\bar{d} \ind^{<_i} _X g\bar{d}$ for all $i=1,...,n$. Hence, by the previous lemma, we have $\bar{d} \ind _X g\bar{d}$. 

Now fix $m$ and assume $\bar{b}$ realises $p$ $\bar{b} \ind^{<_i}_X g\bar{b}$ for all $i=1,...,m-1$. List $\bar{b}$ as $b_1>_m b_2>_m ... >_m b_l$. Inductively on $k$, we change $b_k$ wherever necessary. For the inductive base, we let $c_1 =b_1$. Let $\bar{b}^{k}=(b_1,...,b_{k})$ and $\hat{b}^k=(b_{k+1},...,b_l)$. For the inductive hypothesis, suppose we now have $\bar{c}^{k}=(c_1,...,c_{k})$ satisfying the followings:
\begin{enumerate}[(i)]
\item $tp(\bar{c}^k \hat{b}^k /X)=tp(\bar{b}/X)=p$,
\item $\bar{c}^k \ind^{<_m}_X g\bar{c}^k$, and
\item $tp^{<_i}(\bar{c}^k \hat{b}^kg(\bar{c}^k\hat{b}^k)/X )=tp^{<_i}(\bar{b} g\bar{b}/ X)$ for all $i=1,...,m-1$
\end{enumerate}
We want to choose $c_{k+1}$ so that $\bar{c}^{k+1}=(\bar{c}^{k},c_{k+1})$ satisfies (i)-(iii) for $k+1$. We divide into the following two cases.

Case I. If $b_{k+1} \ind^{<_m} _X g\bar{c}^{k}$, let $c_{k+1}=b_{k+1}$. Then we automatically have (i) and (iii) for $\bar{c}^{k+1}=(\bar{c}^k,c_{k+1})$. For (ii), since $b_{k+1} \ind^{<_m} _X g\bar{c}^{k}$, $c_{k+1}=b_{k+1}$ and $\bar{c}^k \ind^{<_m}_X g\bar{c}^k$, by the previous lemma, we have $\bar{c}^k c_{k+1}\ind^{<_m}_X g\bar{c}^k$.  Since $gc_{k+1} <_m c_{k+1}<_m c_k<_m ....<_m c_1$, we also have $\bar{c}^k c_{k+1}\ind^{<_m}_X gc_{k+1}$. Therefore, by the previous lemma, we have $\bar{c}^k c_{k+1}\ind^{<_m}_X g(\bar{c}^kc_{k+1})$.

Case II. $b_{k+1} \ind^{<_m} _X g\bar{c}^{k}$ does not hold, i.e. there exists $c_i \in \bar{c}^k$ such that $b_{k+1} <_m  gc_i$ and there does not exist $x \in X$ such that $b_{k+1} <_m x<_m  gc_i$. Let $x_1,x_2$ be the lower and upper bound of $b_{k+1}$ with respect to $X$ on $<_m$. This means that there exist $c_{j-1},c_j\in\bar{c}$ such that $x_1<_mb_{k+1}<_mgc_j<_mx_2<_mgc_{j-1}$. We also have $gc_j <_mc_k$ as otherwise, as $\bar{c}^k \ind^{<_m}_X g\bar{c}^k$, there would exist $x'\in X$ such that $c_k <_mx<_m gc_j$. Since $b_{k+1}<_m c_k$ from $tp(\bar{c}^k \hat{b}^k /X)=tp(\bar{b}/X)$, we have $x_1 <_m b_{k+1}<_m c_k <_m x' <_m gc_j <_m x_2$. This contradicts that $x_1,x_2$ are the lower and upper bound of $b_{k+1}$ with respect to $X$ on $<_m$. So, the type
\begin{align*}
 q:=&\cup_{i=1}^{m-1} tp^{<_i}(b_{k+1}/\bar{c}^k\hat{b}^{k+1}Xg(\bar{c}^k\hat{b}^{k+1})g^{-1}(\bar{c}^k\hat{b}^{k+1}X )) \\
&\cup \{ gc_j<_m x<_m \min \{x_2,c_k\} _m\} \\
&\cup _{i=m+1}^n tp(b_{k+1} /\bar{c}^k\hat{b}^{k+1} X).\\
\end{align*}
is consistent since it is consistent on each $<_i$. Hence, we can choose $c_{k+1}$ realising $q$.
\begin{figure}[h]
\centering
\begin{tikzpicture}
  \draw (0,0) node (1){}

++(1.2,0) node[draw,circle,fill=black,minimum size=4pt,inner sep=0pt]  {}
 ++(0,0) node[above] {$x_1$}
     ++(1.2,0) node[draw,circle,fill=black,minimum size=4pt,inner sep=0pt] {} 
   ++(0,0) node[above] {$b_{k+1}$}
   ++(1.2,0) node[draw,circle,fill=black,minimum size=4pt,inner sep=0pt] {}
++(0,0.2) node[above] {$...$}
   ++(1.2,-0.2) node[draw,circle,fill=black,minimum size=4pt,inner sep=0pt] {}
++(0,0) node[above]  {$gc_{j}$}
     ++(1.2,0) node[draw,circle,fill=black,minimum size=4pt,inner sep=0pt] {} 
   ++(0,0) node[above] {$c_{k+1}$}
   ++(1.2,0) node[draw,circle,fill=black,minimum size=4pt,inner sep=0pt]  {}
++(0,0) node[above] {$x_2$}
   ++(1.2,0) node[draw,circle,fill=black,minimum size=4pt,inner sep=0pt] {}
++(0,0) node[above] {$gc_{j-1}$}
++(0.8,0) node (2){}
++(1,-0.3) node[above] {$<_m$};
\draw[-] (1) -- (2) node[draw=none,fill=none]{};
\end{tikzpicture}
\end{figure}

We show that $\bar{c}^{k+1}=(\bar{c}^{k},c_{k+1})$ satisfies (i)-(iii) in the inductive hypothesis:
\begin{enumerate}[(i)]
\item By construction, we have $b_l<_m...<_m   b_{k+1} <_m c_{k+1}<_m \min_m\{x_2,c_k\}\\ \leq _m c_k<_m ...<_m c_1$ and $x_1,x_2$ are the lower and upper constraint of $b_{k+1},c_{k+1}$ with respect to $X$ on $<_m$. So, we have 
\[tp^{<_m}(\bar{c}^{k}c_{k+1} \hat{b}^{k+1} /X)=tp^{<_m}(\bar{c}^{k}b_{k+1}\hat{b}^{k+1} /X).\] 
By the inductive hypothesis (i), we have 
\[tp^{<_m}(\bar{c}^{k+1} \hat{b}^{k+1} /X)=tp^{<_m}(\bar{c}^{k}b_{k+1} \hat{b}^{k+1} /X)=tp^{<_m}(\bar{b}/X).\]

\item Since $gc_j<_m c_{k+1} <_m \min\{x_2,c_k\}\leq _m x_2 <_m gc_{j-1}$, we have that for all $gc_i >_mc_{k+1}$, there exists $x_2 \in X$ such that $gc_i >_m x_2>_mc_{k+1}$. Hence, $c_{k+1} \ind^{<_m}_X g\bar{c}^k$ and thus, by the previous lemma and the inductive hypothesis that $\bar{c}^k \ind^{<_m}_X g\bar{c}^k$, we have $\bar{c}^k c_{k+1}\ind^{<_m}_X g\bar{c}^k$. Since $gc_{k+1} <_m c_{k+1}<_m c_k<_m ....<_m c_1$, we also have $\bar{c}^k c_{k+1}\ind^{<_m}_X gc_{k+1}$. Therefore, by the previous lemma, we have $\bar{c}^k c_{k+1}\ind^{<_m}_X g(\bar{c}^kc_{k+1})$, which is the same as $\bar{c}^{k+1}\ind^{<_m}_X g\bar{c}^{k+1}$

\item  For any $i=1,...,m-1$, since $c_{k+1}$ realises $q$, we have
\[ tp^{<_i}(c_{k+1}/\bar{c}^k \hat{b}^{k+1}g(\bar{c}^k\hat{b}^{k+1}) X  )=tp^{<_i}(b_{k+1}/\bar{c}^k\hat{b}^{k+1} g(\bar{c}^k\hat{b}^{k+1}) X).\]
By the inductive hypothesis, we also have that
\[ tp^{<_i}(b_{k+1}/\bar{c}^k\hat{b}^{k+1} g(\bar{c}^k\hat{b}^{k+1}) X)=tp^{<_i}(b_{k+1}/\bar{b}^k \hat{b}^{k+1}g(\bar{b}^k\hat{b}^{k+1}) X).\]
Hence, we have 
\[ tp^{<_i}(c_{k+1}/\bar{c}^k\hat{b}^{k+1} g(\bar{c}^k\hat{b}^{k+1}) X)=tp^{<_i}(b_{k+1}/\bar{b}^k\hat{b}^{k+1} g(\bar{b}^k\hat{b}^{k+1}) X).\]

We also have that
\begin{align*}
&tp^{<_i}(gc_{k+1}/\bar{c}^k \hat{b}^{k+1}g(\bar{c}^k\hat{b}^{k+1}) X ) &\\
=&g\cdot tp^{<_i}(c_{k+1}/ g^{-1}(\bar{c}^k \hat{b}^{k+1}X)\bar{c}^k \hat{b}^{k+1}))&\\
=&g\cdot tp^{<_i}(b_{k+1}/g^{-1}(\bar{c}^k \hat{b}^{k+1}X)\bar{c}^k \hat{b}^{k+1}))& \text{since $c_{k+1}$ realises $q$} \\
=& tp^{<_i}(gb_{k+1}/\bar{c}^k\hat{b}^{k+1} g(\bar{c}^k\hat{b}^{k+1}) X )&\\
=&tp^{<_i}(gb_{k+1}/\bar{b}^k\hat{b}^{k+1} g(\bar{b}^k\hat{b}^{k+1}) X)& \text{by the inductive hypothesis (iii)}
\end{align*}
Since $gc_{k+1}<_ic_{k+1}$ and $gb_{k+1}<_ib_{k+1}$, we have
\[  tp^{<_i}(c_{k+1}gc_{k+1}/\bar{c}^k\hat{b}^{k+1} g(\bar{c}^k\hat{b}^{k+1}) X)=tp^{<_i}(b_{k+1}gb_{k+1}/\bar{b}^k \hat{b}^{k+1}g(\bar{b}^k\hat{b}^{k+1}) X).\]
By rearranging the type, we have shown 
\[tp^{<_i}(\bar{c}^{k+1} \hat{b}^{k+1}g(\bar{c}^{k+1}\hat{b}^{k+1})/X)=tp^{<_i}(\bar{b} g\bar{b}/ X).\]
\end{enumerate}

Hence by induction, we can find $\bar{c}$ realising $tp(\bar{b} /X)=p$ such that $tp^{<_i}(\bar{c} g\bar{c}/X)=tp^{<_i}(\bar{b} g\bar{b}/ X)$ for $i=1,...,m-1$ and $\bar{c} \ind^{<_n}_X g\bar{c}$. For all $i=1,...,m-1$, since $\bar{b} \ind^{<_i}_X g\bar{b}$, we also have $\bar{c} \ind^{<_i}_X g\bar{c}$ by Invariance. Hence, inductively, we can find $\bar{d}$ realising $p$ such that $\bar{d} \ind^{<_i} _X g\bar{d}$ for all $i=1,...,n$. Therefore, by the previous lemma, we have $\bar{d} \ind _X g\bar{d}$. Hence, $g$ moves almost $R$-maximally.

By a symmetric argument, we can show that if $ga >_{i} a$ for any $a \in \M_n$ and $i=1,...,n$, then $g$ moves almost $L$-maximally. 

\end{proof}

Note that the results we have shown about $\M_n$ also holds for $\M_n^I$ since $\M_n \cong \M_n^I$.

\begin{proof}[Proof of part (ii) of Theorem \ref{main}]
Let $\M_n$ be the universal $n$-linear order for $n\geq 2$. Let $g$ be a non-trivial automorphism of $\M_n$. Then, by Theorem \ref{universallinearorder}, there exists $I\subseteq \{1,...,n\}$ and $h\in Aut(\M_n^I)$ such that $h$ can be written as a product of conjugates of $g$ and $g^{-1}$ and $h$ has a single $+$-orbital on $<_j$ for any $j=1,...,n$ in $\M_n^I$. Then $h^{-1}$ has a single $-$-orbital on $<_j$ for any $j=1,...,n$ in $\M_n^I$. By the previous theorem, $h$ moves almost $R$-maximally and $h^{-1}$ moves almost $L$-maximally. Therefore, by Theorem \ref{swirthm}, any element of $Aut(\M_n^I)$ can be written as a product of conjugates of $g$ and $g^{-1}$. Since $Aut(\M_n)=Aut(\M_n^I)$, any element of $Aut(\M_n^I)$ can be written as a product of conjugates of $g$ and $g^{-1}$. Thus, $Aut(\M_n)$ is simple.
\end{proof}

\printbibliography

\Addresses

\end{document}